\documentclass[reqno]{amsart}
\usepackage{amsfonts}
\usepackage{amssymb}
\usepackage{times}
\usepackage{xcolor}
\usepackage{xfrac}

\usepackage{tasks}
\usepackage{xcolor}
\usepackage{tkz-graph}
\usepackage{hyperref}
\usepackage{tikz,pgf}
\usepackage{cite}

\newcommand{\CC}{\mathbb{C}}
\newcommand{\KK}{\mathbb{K}}
\newcommand{\ZZ}{\mathbb{Z}}

\def\gg{{\mathfrak g}}

\def\aa{{\mathfrak a}}
\def\cal{\mathcal }

\makeatletter
\def\Ddots{\mathinner{\mkern1mu\raise\p@
\vbox{\kern7\p@\hbox{.}}\mkern2mu
\raise4\p@\hbox{.}\mkern2mu\raise7\p@\hbox{.}\mkern1mu}}
\makeatother

\theoremstyle{plain}%

 \newtheorem{thm}{Theorem}[section]
 \newtheorem{cor}{Corollary}[section]
 \newtheorem{lem}{Lemma}[section]
 \newtheorem{prop}{Proposition}[section]
 \theoremstyle{definition}
 \newtheorem{defn}{Definition}[section]
 \newtheorem{exa}{Example}[section]
 \theoremstyle{remark}
 \newtheorem{rem}{Remark}[section]
 \numberwithin{equation}{section}

\newfont{\hueca}{msbm10}

\begin{document}

\title{Quadratic Symplectic Lie superalgebras with a filiform module as an odd part}

\thanks{The first author was supported by the Centre for Mathematics of the University of Coimbra - UIDB/00324/2020, funded by the Portuguese Government through FCT/MCTES. Third author was supported  by Agencia Estatal de Investigaci\'on (Spain), grant PID2020-115155GB-I00 (European FEDER support included, UE). Fourth author was supported by the PCI of the UCA `Teor\'\i a de Lie y Teor\'\i a de Espacios de Banach' and by the PAI with project number FQM298.}

\author[Elisabete~Barreiro]{Elisabete~Barreiro}
\address{Elisabete~Barreiro, University of Coimbra, CMUC, Department of Mathematics,  FCTUC,
Largo D. Dinis
3000-143 Coimbra, Portugal.\\
{\em E-mail}: {\tt mefb@mat.uc.pt}}{}

\author[Sa\"{i}d Benayadi]{Sa\"{i}d Benayadi}
\address{Sa\"{i}d Benayadi,
Laboratoire de Math\'{e}matiques IECL UMR CNRS 7502,
Universit\'{e} de Lorraine, 3 rue Augustin Fresnel, BP 45112,
F-57073 Metz Cedex 03, France.
\newline
{\em E-mail}: {\tt said.benayadi@univ-lorraine.fr}}{}

\author[Rosa M. Navarro]{Rosa M. Navarro}
\address{Rosa M. Navarro,
Departamento de Matem{\'a}ticas, Universidad de Extremadura, C{\'a}ceres, Espa\~na.\\ {\em E-mail}: {\tt rnavarro@unex.es}}{}

\author[Jos\'{e} M. S\'{a}nchez]{Jos\'{e} M. S\'{a}nchez}
\address{Jos\'{e} M. S\'{a}nchez, Departamento de Matem\'aticas, Universidad de C\'adiz, Campus de Puerto Real, 11510, Puerto Real, C\'adiz, Espa\~na.\\ {\em E-mail}: {\tt txema.sanchez@uca.es}}{}

\begin{abstract}
The present work studies deeply quadratic symplectic Lie superalgebras, obtaining, in particular, that they are all nilpotent. Consequently, we provide classifications in low dimensions and identify the double extensions that maintain symplectic structures. By means of both elementary odd double extensions and generalized double extensions of quadratic symplectic Lie superalgebras, we obtain an inductive description of quadratic symplectic Lie superalgebras of filiform type.


\end{abstract}

\maketitle

\section{Introduction}

The theory of quadratic Lie algebras has gained much attention in recent years, particularly since they have arisen in a number of other disciplines such as geometry and physics. Moreover, quadratic Lie algebras play a crucial role in conformal field theory, and with respect to these algebras, there is a Sugawara construction that exists specifically in the case of quadratic Lie algebras
\cite{11}. Nowadays, finite-dimensional quadratic Lie algebras over a field of characteristic zero are well known (see \cite{4, 8, 10, 12, 15}). The procedure of double extension was introduced by Medina
and Revoy in \cite{15}, which enabled them to give a certain description of
quadratic Lie algebras. More precisely, it is showed that every quadratic Lie algebra can be constructed as a direct sum of irreducible ones, and the latter by a sequence of double extensions. Likewise, the $T^*$-extension, i.e. a construction that generalizes the semidirect product, was given in \cite{8} to describe all solvable quadratic Lie algebras. A different approach, on the other hand, was given recently by Kath and Olbrich in \cite{14*} and the first attempt to generalize it to the context of  quadratic Lie superalgebras can be found in \cite{Doubleextension}. Note that it is more difficult to apply the concept of double extension to classify inductively quadratic Lie superalgebras than to describe inductively quadratic Lie algebras. There are
a number of reasons for this, some of which include a $1$-dimensional subspace of a Lie superalgebra ${\mathfrak g}$ is not necessarily a Lie subsuperalgebra of ${\mathfrak g}$  and  there is no analog to the Lie Theorem for solvable Lie superalgebras. In spite of the difficulties, the authors in \cite{Doubleextension} showed that every irreducible quadratic Lie superalgebra whose centre intersects the even part in a non-trivial way is a double extension. Nevertheless, such a condition is rather restrictive. Later, in \cite{9}, the structure of the Lie superalgebras with reductive even part and the action of the even part on the odd part completely reducible is studied. Afterward,
Benayadi in \cite{5} obtained an inductive description of quadratic Lie superalgebras ${\mathfrak g} = {\mathfrak g}_{\bar 0} \oplus {\mathfrak g}_{\bar 1}$ such that ${\mathfrak g}_{\bar 0}$ is a reductive Lie algebra, and the action of ${\mathfrak g}_{\bar 0}$ on ${\mathfrak g}_{\bar 1}$ is completely reducible. Additionally, the same author  gave an inductive description of the solvable quadratic Lie superalgebras such that the odd part is a completely reducible module on the even part. More recently, the authors in \cite{2} generalized the notion of double extension of quadratic Lie superalgebras in order to present an inductive description of solvable quadratic Lie superalgebras. In particular, it was showed that all solvable quadratic Lie superalgebras are obtained by a sequence of generalized double extensions by $1$-dimensional Lie superalgebras. In \cite{JPAA(2009-ELS)} were presented some non-trivial examples of quadratic Lie superalgebras such that the even part is a reductive Lie algebra and the action of the even part on the odd part is not completely reducible, and was given an inductive description of this class of quadratic Lie superalgebras. More recently, in \cite{quadraticFLSA}, the authors characterized in any dimension,  via double extensions, a very special family of quadratic Lie superalgebras ${\mathfrak g}={\mathfrak g}_{\bar 0}\oplus {\mathfrak g}_{\bar 1}$ such that ${\mathfrak g}_{\bar 0}$ is a reductive Lie algebra and ${\mathfrak g}_{\bar 1}$ is not a completely reducible ${\mathfrak g}_{\bar 0}$-module. In particular, these superalgebras verify that ${\mathfrak g}_{\bar 1}$ is a filiform ${\mathfrak g}_{\bar 0}$-module (filiform type). They showed that the study of quadratic Lie superalgebras of filiform type can be reduced to those that are solvable and  obtained an inductive description of solvable quadratic Lie superalgebras of filiform type via both double extensions and odd double extensions of quadratic ones.

The purpose of our study is to investigate quadratic Lie superalgebras with symplectic structures. Prior to that, however, we analyzed quadratic symplectic Lie superalgebras in general and obtained some significant results. More specifically, we answered some open questions raised in \cite{simplecticLSA}.

\section{Basic concepts and previous results}

Throughout this work, we consider an arbitrary field of characteristic zero.

\subsection{Basic concepts and some results}

We shall establish in this section some definitions needed later on. For an integer $i$ we denote by $\bar{i}$ its correspondent equivalence class in $\mathbb{Z}_2 = \{\bar{0}, \bar{1}\}$. For two integers $i, j$  we use the well-defined notation $(-1)^{\overline{ij}}$ for $(-1)^{ij} \in \{-1, 1\}$. For a $\mathbb{Z}_2$-graded vector space $V = V_{\bar 0} \oplus V_{\bar 1}$ over the field $\KK$, as usual, we write  $V_{\bar 0}$ for its {\em even part} and $V_{\bar 1}$ for its {\em odd part}. An element $X$ of $V$ is called {\it homogeneous} if  either  $X \in V_{\bar 0}$ or $X \in V_{\bar 1}$. In this work, all elements will be supposed to be homogeneous unless indicated otherwise. A linear map $\phi : V \to W$ between two $\mathbb{Z}_2$-graded vector spaces is called {\it even} if $\phi(V_{\bar 0}) \subseteq W_{\bar 0}$ and $\phi(V_{\bar 1}) \subseteq W_{\bar 1}$. While, it is called {\it odd} if $\phi(V_{\bar 0}) \subset W_{\bar 1}$ and $\phi(V_{\bar 1}) \subset W_{\bar 0}$. Clearly, ${\rm Hom}(V,W) = {\rm Hom}(V,W)_{\bar 0} \oplus {\rm Hom}(V,W)_{\bar 1}$, where the first summand comprises all the even linear maps, and the second  all the odd.  Tensor products $V \otimes W$ are $\mathbb{Z}_2$-graded vector spaces, where its even part is   $(V \otimes W)_{\bar 0} := (V_{\bar 0} \otimes W_{\bar 0}) \oplus (V_{\bar 1} \otimes W_{\bar 1})$ and odd part $(V \otimes W)_{\bar 1} := (V_{\bar 0} \otimes W_{\bar 1}) \oplus (V_{\bar 1} \otimes W_{\bar 0})$.
We shall write $\displaystyle X \in {\gg}_x$ to mean that $X$ is a homogeneous element of the Lie
superalgebra $\gg$ of degree $x$, with $x \in \mathbb{Z}_2$.
\begin{defn}
A {\it Lie superalgebra} is a $\mathbb{Z}_2$-graded vector space $\gg = \gg_{\bar 0} \oplus \gg_{\bar 1}$, with an even bilinear operation $[\cdot,\cdot]$, which satisfies the conditions:
\begin{enumerate}
\item [i.] $[X,Y] = -(-1)^{xy}[Y,X]$,
\item [ii.] $(-1)
^{xz} [X,[Y,Z]] + (-1)^{xy} [Y,[Z,X]] + (-1)^{yz} [Z, [X,Y]] = 0$, ({\it super Jacobi identity})
\end{enumerate}
for any $X \in \gg_{x}, Y \in \gg_{y}, Z \in \gg_{z}$, with $x,y,z \in \mathbb{Z}_2$.
\end{defn}

\noindent The general background on Lie superalgebras can be found in \cite{16}.
From the previous definition, $\gg_{\bar 0}$ is a Lie algebra and $\gg_{\bar 1}$ is a $\gg_{\bar 0}$-module. The Lie superalgebra structure also contains the
symmetric pairing $S^2 \gg_{\bar 1} \to \gg_{\bar 0}$, which is a $\gg_{\bar 0}$-morphism and satisfies the super Jacobi identity applied to three
elements of $\gg_{\bar 1}$. 

The definition below gathers some properties of bilinear forms on Lie superalgebras.
\begin{defn}\label{def_22}
Let $\gg$ be a Lie superalgebra. A bilinear form $B$ on $\gg$ is
\begin{enumerate}
\item[(i)]
\textit{supersymmetric} if $ B(X,Y)=(-1)^{xy}B(Y,X)$, for any $\: {X \in \gg_x , Y \in \gg_y}$;
\item[(ii)] \textit{skew-supersymmetric}   if $B(X,Y)=-(-1)^{xy}B(Y,X)$, for $\: {X \in \gg_x , Y \in \gg_y}$;
\item[(iii)]  \textit{non-degenerate} if $X \in \gg$ satisfies $B(X,Y)=0$, for any $\: { Y \in \gg}$, then $X=0$;
\item[(iv)] \textit{invariant}  if $B([X,Y],Z)=B(X,[Y,Z])$, for $\: X,Y,Z \in \gg$;
\item[(v)] \textit{even} if $ B(\gg_{\bar 0},\gg_{\bar 1}) = B(\gg_{\bar 1},\gg_{\bar 0}) = \{0\}$;
\item[(vi)] \textit{scalar $2$-cocycle} on $\gg$ if for any $\: X \in \gg_x , Y \in \gg_y, Z \in \gg_z$,
\begin{eqnarray}
\displaystyle (-1)^{xz}B (X,[Y,Z])+(-1)^{yx}B
(Y,[Z,X])+(-1)^{zy}B (Z,[X,Y])&=&0. \nonumber 
\end{eqnarray}
The set of all scalar $2$-cocycles on $\gg$ is denoted by ${\mathcal Z}^2(\gg,\KK)$;
\item[(vii)] 
\textit{scalar $2$-coboundary} on $\gg$ if there exists $\xi $ in the dual of $  \gg $ such that for all $\: X , Y \in \gg $,
\begin{eqnarray}
\displaystyle  B(X,Y)=\xi ([X,Y]).  \nonumber 
\end{eqnarray}
The set of all scalar $2$-coboundaries on $\gg$ is
denoted by ${\mathcal B}^2(\gg,\KK)$.
\end{enumerate}
\end{defn}

\begin{rem}
There are some authors, see \cite{A,B}, who refer to supersymmetric and skew-symmetric in Definition \ref{def_22} as symmetric and antisymmetric, respectively, for reasons explained in the text.
\end{rem}

\begin{defn}
A Lie superalgebra $\gg$ is called
\begin{enumerate}
\item[(i)] \textit{quadratic} \index{Lie superalgebra!quadratic} if there exists a bilinear form $B$ on $\gg$ such that $B$ is even, supersymmetric, non-degenerate, and invariant. It is denoted by $(\gg,B)$ and $B$ is called an \textit{invariant scalar product} on $\gg$.
\item [(ii)]    \textit{symplectic} \index{Lie superalgebra!symplectic} if there exists a bilinear form $\omega$ on $\gg$ such that it is an even, skew-supersymmetric, non-degenerate, and scalar $2$-cocycle on $\gg$. In this case, it is denoted by $(\gg,\omega)$ and $\omega$ is said a \textit{symplectic structure} on $\gg$.
\item[(iii)] \textit{quadratic symplectic} \index{Lie superalgebra!quadratic symplectic} if $\gg$ is equipped with an invariant scalar product $B$ and a symplectic structure $\omega$. We denote it by $(\gg,B,\omega)$.
\end{enumerate}
\end{defn}

\begin{rem}
In the literature, symplectic Lie (super)-algebras are sometimes referred to as quasi-Frobenius, see, e.g., \cite{Sofiane, E, F2}.
\end{rem}

\noindent A {\it superderivation} $D$ of degree $d$ of ${\mathfrak g}$, with $d \in \mathbb{Z}_2$,
is a linear map on ${\mathfrak g}$ homogeneous of degree
$d$ that satisfies $$D([X,Y]) = [D(X),Y] + (-1)^{dx}[X,D(Y)]$$ for any $X \in \gg_x , Y \in \gg$. Let ${\rm Der}({\mathfrak g}) = {\rm Der}({\mathfrak g})_{\bar 0} \oplus {\rm Der}({\mathfrak g})_{\bar 1}$ be the space of all derivations on ${\mathfrak g}$.

\begin{defn} Let ${\mathfrak g}$ be a Lie superalgebra and $B$ a  bilinear form on $\gg$. A homogeneous superderivation $D$ on $\gg$ of degree $d$ is called \emph{skew-supersymmetric} if $$B(D(X),Y)=-(-1)^{dx}B(X,D(Y)), $$ for any $X \in \gg_x,  Y \in \gg$.
We denote by ${\rm Der}_a(\gg)$ the subspace of ${\rm Der}(\gg)$ spanned by all the homogeneous skew-supersymmetric superderivations of $\gg$. 
\end{defn}

\begin{defn} Let $\omega : {\mathfrak g} \times {\mathfrak g} \to  {\mathbb K}$  be a non-degenerate  bilinear form. To each superderivation $D \in {\rm Der}({\mathfrak g})_d$ we can assign a unique linear map $D^*$, called the {\it adjoint of} $D$, that satisfies $$\omega(D(X),Y) = (-1)^{xd}\omega(X,D^*(Y)),$$
\end{defn}
\noindent for any $X \in \gg_x , Y \in \gg$. It is clear to see that $D^*$ is also a superderivation.

\begin{defn}
 Given a Lie superalgebra $\gg$, a $\mathbb{Z}_2$-graded vector space $A$ and $\pi : \gg \to {\mathfrak g}{\mathfrak l}(A)$
a representation of $\gg$ on $A$.
We recall that $A$ is a $\gg${\it -module} and we denote $  Xa := \pi(X)(a)$, for $X\in \gg$ and $a\in A$.
\end{defn}

\noindent If $(\gg = \gg_{\bar 0} \oplus \gg_{\bar 1}, [\cdot,\cdot])$ is a Lie superalgebra we denote $\rho_g : \gg_{\bar 0} \to {\mathfrak g}{\mathfrak l}(\gg_{\bar 1})$ the representation of $\gg_{\bar 0}$ in $\gg_{\bar 1}$ defined by $\rho_g(X)(Y) := [X,Y]$ for all $X \in \gg_{\bar 0}, Y \in \gg_{\bar 1}$.
The following concept of filiform module was firstly introduced in  \cite{JGP} for defining filiform Lie superalgebras.

\begin{defn}
Let $\gg=\gg_{\bar 0}\oplus \gg_{\bar 1}$ be a Lie superalgebra with $\dim(\gg_{\bar 1})=m > 0$. We said that $\gg_{\bar 1}$ has the structure of \textit{filiform $\gg_{\bar 0}$-module} if the action of $\gg_{\bar 0}$ on $\gg_{\bar 1}$ defines a {\it flag}, that is, a decreasing subsequence of vector subspaces in its underlying vector space $\gg_{\bar 1}$, $$\gg_{\bar 1}=V_m \supset \dots \supset V_1
\supset V_0,$$ with $\dim(V_i)=i$ and such that $[\gg_{\bar 0},V_{i+1}]= V_i$, for $i \in \{0,\dots,m-1\}$. To abbreviate, in the sequel we will refer to $\gg=\gg_{\bar 0} \oplus \gg_{\bar 1}$ as a Lie superalgebra of \textit{filiform type}.
\end{defn}

We recall now the {\em descending central sequence} of a Lie superalgebra $\gg = \gg_{\bar 0} \oplus \gg_{\bar 1}$ which is defined in the same way as for Lie algebras, that is, $${\cal C}^0(\gg): = \gg, \hspace{0.2cm} {\cal C}^{k+1}(\gg) := [{\cal C}^k(\gg),\gg]$$ for $k \geq 0$. In case ${\cal C}^k(\gg)=\{0\}$ for some $k \in \mathbb{N}$, then the Lie superalgebra $\gg  $ is
called {\em nilpotent}. Nevertheless, there are also defined two others descending sequences denoted by ${\cal C}^k(\gg_{\bar 0})$ and ${\cal C}^k(\gg_{\bar 1})$ which will be also important in our study. They are defined as ${\cal C}^0(\gg_{\bar i}) := \gg_{\bar i}, {\cal C}^{k+1}(\gg_{\bar i}) := [\gg_{\bar 0}, {\cal C}^k(\gg_{\bar i})],$ for $k \geq 0, {\bar i} \in \ZZ_2.$

\begin{defn}
Let $\gg=\gg_{\bar 0} \oplus \gg_{\bar 1}$ be a solvable Lie superalgebra. We say that $\gg$ has {\em super-nilindex} or $s${\em -nilindex} $(p,q)$ if it satisfies $${\cal C}^{p-1}(\gg_{\bar 0}) \neq \{0\}, \qquad {\cal C}^{q-1}(\gg_{\bar 1}) \neq \{0\}, \qquad {\cal C}^p(\gg_{\bar 0}) = {\cal C}^q(\gg_{\bar 1})=\{0\}.$$
\end{defn}

\noindent The following result establishes an equivalent definition of filiform module structure.

\begin{lem}  \cite{quadraticFLSA}
If $\gg = \gg_{\bar 0} \oplus \gg_{\bar 1}$ is a solvable Lie superalgebra with {\em super-nilindex} $(-,\dim(\gg_{\bar 1})),$ then $\gg_{\bar 1}$ has filiform $\gg_{\bar 0}$-module structure.
\end{lem}

\begin{rem}
In particular whenever we have a nilpotent action over $\gg_{\bar 1}$ the sequence ${\cal C}^0(\gg_{\bar 1}) := \gg_{\bar 1}, {\cal C}^{k+1}(\gg_{\bar 1}) := [\gg_{\bar 0}, {\cal C}^k(\gg_{\bar 1})],$ $k\geq 0$, defines a {\it filtration} over $\gg_{\bar 1}$, i.e. a decreasing sequence of submodules ${\cal C}^0(\gg_{\bar 1}) \supset {\cal C}^1(\gg_{\bar 1}) \supset {\cal C}^2(\gg_{\bar 1}) \dots$ 
\end{rem}

\begin{lem}\label{lem} \cite{quadraticFLSA} 
A necessary and sufficient condition to get a filiform $\mathfrak g_{\bar 0}$-module structure over the odd part of a Lie superalgebra $\mathfrak g$ is $\dim \left( {\cal C}^{i-1}(\mathfrak g_{\bar 1})/{\cal C}^i(\mathfrak g_{\bar 1}) \right) = 1$ for $1 \leq i \leq \dim(\mathfrak g_{\bar 1})$.
\end{lem}

\noindent Henceforth, a usual,  we will denote the center of $\gg$ by $\mathfrak{z}(\gg) := \{X \in \gg : [X,\gg]=0\}$.

\begin{lem}
\label{lem: V1z(g)} \cite{quadraticFLSA} Let  $({\mathfrak g}={\mathfrak g}_{\bar 0}\oplus {\mathfrak g}_{\bar 1},B)$ be a quadratic Lie superalgebra with $dim({\mathfrak g}_{\bar 1})=m > 0$ such that ${\mathfrak g}_{\bar 1}$ has the structure of  filiform ${\mathfrak g}_{\bar 0}$-module concerning to the flag ${\mathfrak g}_{\bar 1}=V_m \supset \dots \supset V_1 \supset V_0$. 
Then $\displaystyle \mathfrak z(\mathfrak g) $ intersects $ {\mathfrak g}_{\bar 1}$. More precisely,  $\displaystyle V_1 \subseteq \mathfrak z(\mathfrak g)  \cap {\mathfrak g}_{\bar 1}$. 
\end{lem}


\begin{lem} \label{lem: V1Vi} \cite{quadraticFLSA} Let $({\mathfrak g}={\mathfrak g}_{\bar 0}\oplus {\mathfrak g}_{\bar 1},B)$ be a quadratic Lie superalgebra with $dim({\mathfrak g}_{\bar 1})=m > 0$ such as ${\mathfrak g}_{\bar 1}$ has the structure of filiform ${\mathfrak g}_{\bar 0}$-module pertaining to the flag $\displaystyle {\mathfrak g}_{\bar 1} = V_m \supset \dots \supset V_1 \supset V_0$. Then $\displaystyle V_1$ is orthogonal to $\displaystyle  V_i$ with respect to $B$, wherever $1 \leq i \leq m-1$.
\end{lem}

\subsection{Double extension of quadratic   Lie superalgebras }

We will next review the concept of double extensions of Lie superalgebras.

\begin{thm}\cite[Theorem 1]{Doubleextension} \label{thm_Doubleextension}
Let $({\mathfrak a}, B_1)$ be a quadratic Lie superalgebra, ${\mathfrak h}$ a Lie superalgebra and $\psi : {\mathfrak h} \to {\rm Der}_a({\mathfrak a})$ a morphism of Lie superalgebras.
Let $\varphi$ be the map from ${\mathfrak a} \times {\mathfrak a}$ to ${\mathfrak h}^*$, defined by 
$$\varphi(X,Y)(Z) = (-1)^{(x+y)z}B_1(\psi(Z)(X),Y),$$
for any $X \in {\mathfrak a}_x, Y \in {\mathfrak a}_y,  Z \in {\mathfrak h}_z$. Let $\pi$ be the coadjoint representation of ${\mathfrak h}$. Then the vector space ${\mathfrak g}={\mathfrak h} \oplus {\mathfrak a} \oplus {\mathfrak h}^*$ with the product
\begin{align*}
[X_2+X_1+f, Y_2+Y_1+g] &= [X_2,Y_2]_{{\mathfrak h}} + [X_1,Y_1]_{{\mathfrak a}} + \psi(X_2)(Y_1) - (-1)^{xy} \psi(Y_2)(X_1)\\
&+ \pi(X_2)(g)-(-1)^{xy} \pi(Y_2)(f)+\varphi(X_1,Y_1),
\end{align*}
where $X_2+X_1+f$ (resp. $Y_2+Y_1+g$) is homogeneous of degree $x$ (resp. $y$) in $\gg$, is a Lie superalgebra.
 
Moreover, if $\gamma$ is a ${\mathfrak h}$-invariant supersymmetric bilinear form, then the bilinear form $T$, defined on ${\mathfrak g}$ by $$T(X_2+X_1+f, Y_2+Y_1+g) := B_1(X_1,Y_1)+\gamma(X_2,Y_2)+f(Y_2)+(-1)^{xy}g(X_2)$$ where $X_2+X_1+f$ and $Y_2+Y_1+g$ are  homogeneous of respective degrees $x$, $y$, is an invariant scalar product on $\gg$.

The Lie superalgebra $\gg$ is called the double extension of $({\mathfrak a},B_1)$ by ${\mathfrak h}$ (by means of $\psi$).
\end{thm}

Actually, we work with a particular case of double extension of Lie superalgebras which we refer now. Let $\displaystyle (\aa,B)$ be a quadratic Lie superalgebra and $\displaystyle D$ an even skew-supersymmetric superderivation of $\displaystyle (\aa,B)$. We consider the double extension $\displaystyle (\gg =\mathbb{K}e \oplus \aa \oplus \mathbb{K}e^*, \widetilde{B})$ of $\displaystyle (\aa,B)$ by the one-dimensional Lie algebra $\displaystyle \mathbb{K}e$ (by means of $D$). The multiplication on $\gg$ is defined by
\begin{eqnarray*}
\displaystyle \hspace*{0cm}[e,X]  = D(X), &
[X,Y]  =  [X,Y]_{\aa}+B(D(X),Y)  e^*,
& [e^*,\gg] = \{0\},
\end{eqnarray*}
for any $X,Y \in \aa$. Moreover, the invariant scalar product $\displaystyle \widetilde{B}$ on $\gg$ is given by
\begin{eqnarray*}
\displaystyle \widetilde{B}\mid_{\aa \times \aa}\: =\:B, 
&  \widetilde{B}(e^* ,e)\: =\: 1, 
&  \widetilde{B}(\aa,e)=  \widetilde{B}(\aa,e^*)\: =\: \{0\} , 
\widetilde{B}(e,e)=  \widetilde{B}(e^*,e^*)\: =\: 0.
\end{eqnarray*}


\section{Symplectic extensions of nilpotent Lie superalgebras of filiform type}

Here, we will introduce the concept of a symplectic double extension of a symplectic Lie superalgebra
as introduced in \cite{Sofiane}.  

It is a very well-known fact that if $\gg=\gg_{\bar 0}\oplus \gg_{\bar 1}$ is a  non-zero  nilpotent Lie superalgebra then the center of the superalgebra, $\mathfrak z({ \gg})=\mathfrak z({\gg})_{\bar 0} \oplus \mathfrak z({ \gg})_{\bar 1}$, is not trivial. If $\mathfrak z({\gg})_{\bar 1} \neq \{0\}$ then there exists at least an odd element $e \in  ({\mathfrak z}(\gg) \cap \gg_{\bar 1}) \setminus \{0\}$.  Moreover, if we consider a symplectic structure on $\gg$, denoted by $\omega$, and the condition $\omega(e,e)=0$  holds true, then
we can apply the process of $\delta_{\bar 1}$-extension of symplectic Lie superalgebras
as described in \cite{Sofiane}.  Analogously, and for an even central element it can be applied other extensions  of symplectic Lie superalgebras \cite{Sofiane}. 

We will use this framework extensively in this work since, as we will see in the follo\-wing section, all symplectic Lie superalgebras with invertible derivations are always nilpotent.      Further, if we take the quadratic symplectic Lie superalgebras of filiform type, we
will be able to guarantee the existence of such an   $e \in  ({\mathfrak z}(\gg) \cap \gg_{\bar 1} )\setminus \{0\}$  verifying $\omega(e,e)=0$.

\subsection{$\delta_{\bar 1}$-extension  of symplectic Lie superalgebras}
Let $(\gg, \omega)$ be  symplectic Lie superalgebra.
Let $\delta \in {\rm Der}(\gg)_{\bar 1}$ be a derivation such that the map $\Omega : {\mathfrak g} \times {\mathfrak g} \to \mathbb{K}$ defined as
\begin{equation}\label{Def_Omega}
\Omega(X,Y) :=   \omega \bigl((\delta^2 - (\delta^*)^2)(X),Y \bigr),  
\end{equation} $X,Y \in \gg$, is an element of  ${\mathcal B}^2(\gg,\KK)$, i.e., a $2$-coboundary. Since $\omega$ is non-degenerate, there exists $Y_0 \in \gg_{\bar 0}$ such that $\Omega(X,Y) = \omega(Y_0,[X,Y])$, for any $X,Y \in \gg$.

Next, we recall the concept of the $\delta_1$-extension of symplectic Lie superalgebras.

\begin{thm}\cite[Theorem 4.2.1]{Sofiane}\label{Teorema_4.2.1.}
Let $(\gg, \omega)$ be an even symplectic Lie superalgebra. Let  $\delta \in {\rm Der}(\gg)_{\bar 1}$ be a derivation such that Condition \eqref{Def_Omega} is satisfied together with $$\delta^2 = ad_{Y_0} \hspace{0.6cm} {\rm {\it and}} \hspace{0.6cm} \delta(Y_0) = 0.$$ Then, there exists a Lie superalgebra structure on $\widetilde{\gg} := \KK e \oplus \gg \oplus (\KK e)^*$ for $e$ odd, defined for any $X,Y \in \gg$ as
$$[e,\widetilde{\gg}] = 0, \hspace{0.4cm} [X,Y] = [X,Y]_{\gg} + (\omega(\delta(X),Y) + (-1)^x \omega(X,\delta(Y))e,$$
$$[e^*, e^*]  = 2Y_0, \hspace{0.4cm} [e^*, X] = \delta(X) - \omega(X, Y_0)e.$$ There exists a  symplectic form $\widetilde{\omega}$ on $\widetilde{\gg}$ defined as follows: for any $X \in \gg$
$$\widetilde{\omega}|_{\gg\times \gg}  = \omega, \hspace{0.2cm} \widetilde{\omega}(X, e)  = \widetilde{\omega}(X, e^*)  = 0, \hspace{0.2cm} \widetilde{\omega}(e^*,e)  = 1, \hspace{0.2cm} \widetilde{\omega}(e,e)  = \widetilde{\omega}(e^*,e^*)  = 0.$$
The even symplectic Lie superalgebra $\displaystyle (\widetilde{\gg} =\KK e \oplus \gg \oplus \KK e^*,\widetilde{\omega})$ is called an {\it $\delta_{\bar 1}$-extension of} $(\gg, \omega)$ by the $1$-dimensional Lie superalgebra $(\KK e)_{\bar 1}$ (by means of the odd  superderivation $\delta$ and $Y_0$).
\end{thm}

\noindent By \cite[Theorem 4.2.4]{Sofiane}, the converse of Theorem \ref{Teorema_4.2.1.}   is also valid.  That is, let $(\widetilde{\gg},\widetilde{\omega})$ be an symplectic Lie superalgebra. If there exists a non-zero $X \in {\mathfrak z}(\widetilde{\gg})_{\bar 1}$ such that $\widetilde{\omega}(X,X) = 0$ then $(\widetilde{\gg}, \widetilde{\omega})$ is a $\delta_{\bar 1}$-extension of an  symplectic Lie superalgebra $(\gg, \omega)$.

\section{Quadratic Symplectic Lie superalgebras: General case.}

Let us recall the following result (see \cite[Proposition 4.2]{simplecticLSA}). A quadratic Lie superalgebra $(\gg,B)$ is a symplectic Lie superalgebra $(\gg, \omega)$ if and only if there is an even invertible skew-supersymmetric superderivation $\delta$ of $(\gg, B)$ verifying $\omega(X,Y)=B(\delta(X),Y)$ for all $X,Y \in \gg$. 
Moreover, in \cite{simplecticLSA} it was remarked that since the Lie superalgebra admits an invertible even superderivation then ${\mathfrak{g}}_{\bar 0}$ is nilpotent (as it admits an invertible derivation \cite{jacobson}) which leads to the fact that $\gg ={\mathfrak{g}}_{\bar 0}\oplus {\mathfrak{g}}_{\bar 1}$ is solvable \cite{16}.  Furthermore, the following open question was posed:  

\

\noindent {\bf Open question.} \cite {simplecticLSA} Is a quadratic symplectic Lie superalgebra nilpotent? More generally, if a Lie superalgebra $\gg$ admits an even invertible superderivation, is $\gg$ nilpotent?

\

\noindent  The answer to both of these questions is yes, as we shall show next.  

\begin{thm} If a Lie superalgebra $\gg$ admits an even invertible superderivation then $\gg$ is nilpotent.
\end{thm}

\begin{proof} Let $\gg ={\mathfrak{g}}_{\bar 0}\oplus {\mathfrak{g}}_{\bar 1}$ be a Lie superalgebra and $\delta$ an even invertible superderivation of $\gg$. There is no loss of generality in supossing we have an homogeneous basis $\{ X_1, \dots, X_n ,Y_1,\dots, Y_m \}$, where $\{ X_1, \dots, X_n \} $ is a basis of 
${\mathfrak{g}}_{\bar 0}$ and $\{Y_1,\dots, Y_m \}$ is a basis of ${\mathfrak{g}}_{\bar 1}$. Then, the law of $\gg$ will be determined, in particular, by three sets of bracket products, two skew-symmetrical an one symmetrical, i.e. 

$$\gg:\left\{\begin{array}{ll}
[X_i,X_j]=\sum_{k=1}^{n}C_{ij}^{k}X_k, & C_{ij}^{k}=-C_{ji}^{k} \\{}
[X_i,Y_j]=\sum_{k=1}^{m}E_{ij}^{k}Y_k, & E_{ij}^{k}=-E_{ji}^{k} \\{}
[Y_i,Y_j]=\sum_{k=1}^{n}F_{ij}^{k}X_k, & F_{ij}^{k}=F_{ji}^{k} 
 \end{array}\right.$$
verifying all of them the super Jacobi identity. 

Let us consider now the Lie algebra $\overline{\gg}$  considering the
Lie superalgebra $\gg$ by forgetting the super structure. So, let $\{ \overline{X}_1, \dots, \overline{X}_n,\overline{Y}_1,\dots, \overline{Y}_m \}$ be  a basis of $\overline{\gg}$ and with  the following multiplication table
$$\gg:\left\{\begin{array}{ll}
[\overline{X}_i,\overline{X}_j]=\sum_{k=1}^{n}C_{ij}^{k}\overline{X}_k, & \\{}
[\overline{X}_i,\overline{Y}_j]=\sum_{k=1}^{m}E_{ij}^{k}\overline{Y}_k, &  
\end{array}\right.$$
with $C_{ij}^{k}$ and $E_{ij}^{k}$ as above. It can be easily checked that $\overline{\gg}={\overline{\gg}}_{\bar 0}\oplus {\overline{\gg}}_{\bar 1}$ is in fact a Lie algebra, in particular a $\ZZ_2$-graded Lie algebra. By defining $$\overline{\delta}(\overline{X}_i):=\delta(X_i), \ 1 \leq i \leq n \mbox{ and } \overline{\delta}(\overline{Y}_j):=\delta(Y_j), \ 1 \leq j \leq m$$
it can be seen that $\overline{\delta}$ is an invertible derivation of the Lie algebra $\overline{\gg}$ and then $\overline{\gg}$ is a nilpotent Lie algebra  (see \cite{jacobson}).  Therefore, ${\overline{\gg}}_{\bar 0}$ is nilpotent and also the adjoint representation of ${\overline{\gg}}_{\bar 0}$ over ${\overline{\gg}}_{\bar 1}$, that is,  if  we  denote by $\rho_{\overline{\gg}}$ the representation of Lie algebra ${\overline{\gg}}_{\bar 0}$ on ${\overline{\gg}}_{\bar 1}$ defined by: $$\rho_{\overline{\gg}}(\overline{X})(\overline{Y}):= [\overline{X},\overline{Y}], \ \mbox{ for any } \overline{X}\in {\overline{\gg}}_{\bar 0} \mbox{ and } \overline{Y}\in {\overline{\gg}}_{\bar 1}$$
then for every $\overline{X}\in {\overline{\gg}}_{\bar 0}$, $\rho_{\overline{\gg}}(\overline{X}) \in {\mathfrak g}{\mathfrak l}({\overline{\gg}}_{\bar 1})$ is nilpotent. Let us remark that the structure constants $C_{ij}^{k}$ determine the structure of a nilpotent Lie algebra on ${\overline{\gg}}_{\bar 0}$ and $E_{ij}^{k}$ make the representation $\rho_{\overline{\gg}}$ nilpotent. Consequently, ${\gg}_{\bar 0}$ is nilpotent and also the corresponding adjoint representation of ${\gg}_{\bar 0}$ over ${\gg}_{\bar 1}$, which leads to the fact that the Lie superalgebra $\gg ={\mathfrak{g}}_{\bar 0}\oplus {\mathfrak{g}}_{\bar 1}$ is nilpotent \cite{Salgado}.
\end{proof}

\begin{rem}
We can generalize the previous result to odd superderivations. Indeed, if $\delta$ is an odd invertible superderivation of $\gg$, then $\delta^2 = \frac{1}{2}[\delta,\delta]$ is an even invertible superderivation of ${\mathfrak g}$, so ${\mathfrak g}$ is nilpotent. 
\end{rem}

\begin{cor} \label{cor} Any quadratic symplectic Lie superalgebra is nilpotent.
\end{cor}

Next, and thanks to the previous theorem, we explore classifications in low dimension of quadratic Lie superalgebras in order to obtain the corresponding classification of symplectic quadratic Lie superalgebras, see \cite{singularquadratic, classificationLow}.

Let us recall first the concept of i-isomorphism for  Lie superalgebras \cite{classificationLow}.  Thus, we say that two quadratic Lie  superalgebras  $(\gg,B)$ and $({\mathfrak g}',B')$ are {\it isometrically isomorphic} (or {\it i-isomorphic}, for short) if there exists a Lie  superalgebra  isomorphism $A : \gg \to \gg '$ satisfying $B'(A(X), A(Y)) = B(X,Y)$ for all $X$, $Y \in \gg$. In this case, $A$ is called an {\it i-isomorphism}. Therefore, we have the following results.

 We classify some quadratic Lie superalgebras over the base field $\mathbb{C}$. 

\begin{prop} \label{prop3.1} Let $(\gg,B)$ be  a non-abelian quadratic  Lie superalgebra of dimension $4$. We have: \begin{itemize} 

\item[(i)] If $\dim ({\mathfrak{g}}_{\bar 1})=0$ then there is no quadratic symplectic Lie superalgebra.  

\item[(ii)] If $\dim ({\mathfrak{g}}_{\bar 0})=\dim ({\mathfrak{g}}_{\bar 1})=2$ then there exists only one quadratic symplectic Lie superalgebra, up to i-isomorphism, named $\gg_{4,1}^s = {\gg}_{\bar 0} \oplus {\gg}_{\bar 1}$ where ${\mathfrak{g}}_{\bar 0}=span_{\CC}\{X_0,X_1\}$ and ${\mathfrak{g}}_{\bar 1}=span_{\CC}\{Y_1,Y_2\}$ such that the non-zero bilinear form $B(X_0,X_1)=B(Y_2,Y_1)=1$ and the non-trivial Lie super-brackets:
$$\gg_{4,1}^s:\left\{\begin{array}{l}
[X_1,Y_1]=-[Y_1,X_1]=-2Y_2 \\{}
[Y_1,Y_1]=-2X_0
 \end{array}\right.$$

Moreover,  any symplectic  structure $\omega$ on $\gg_{4,1}^s $ is given by
$$\left\{\begin{array}{l}
\omega(X_0,X_1)=2b_1\\ \omega(Y_2,Y_1)=-b_1\\ \omega(Y_1,Y_1)=b_2
\end{array}\right.$$
where $b_1, b_2$ are parameters  in $ \CC,$ with $ b_1 \neq 0$.
\end{itemize} 
\end{prop} 

\begin{proof}
$(i)$ If $\dim ({\mathfrak{g}}_{\bar 1})=0$, then we have the Lie algebra case, and it is a well-known fact that up to dimension 4, a nilpotent quadratic Lie algebra must be abelian \cite{simplecticLA}. Therefore, there is no non-abelian quadratic  symplectic Lie superalgebra.

\

$(ii)$ If $\dim ({\mathfrak{g}}_{\bar 0})=\dim ({\mathfrak{g}}_{\bar 1})=2$, according to \cite{singularquadratic} and also \cite[Subsetion $2.1.$, page $129$]{classificationLow}, there are, up to i-isomorphism, two quadratic Lie superalgebras: $\gg_{4,1}^s$ (described in the statement of the proposition) and $\gg_{4,2}^s$. The latter is expressed by $\gg_{4,2}^s={\mathfrak{g}}_{\bar 0} \oplus {\mathfrak{g}}_{\bar 1}$ where ${\mathfrak{g}}_{\bar 0}=span_{\CC}\{X_0,X_1\}$ and ${\mathfrak{g}}_{\bar 1}=span_{\CC}\{Y_1,Y_2\}$ such that the non-zero bilinear form $B(X_0,X_1)=B(Y_2,Y_1)=1$ and the non-trivial Lie super-brackets:
$$\gg_{4,2}^s:\left\{\begin{array}{l}
 [X_1,Y_1]=-[Y_1,X_1]=-Y_1 \\{}
[X_1,Y_2]=-[Y_2,X_1]=Y_2 \\{}
[Y_2,Y_1]=[Y_1,Y_2]=X_0
 \end{array}\right. $$
 
Let us note that $\gg_{4,2}^s$ after applying the ismorphism given by 
$$X_0=b, \ X_1=a, \ Y_1=\beta, \ Y_2=\alpha$$
is exactly the Lie superalgebra: $[a,\alpha]=\alpha$, $[a,\beta]=-\beta$ and $[\alpha,\beta]=b$, which is noted by $(C^2_{-1}+A)$, Jordan-Wigner quantization, in Backhouse's classification list \cite{Backhouse}. It can be noted also that $\gg_{4,1}^s$ and the Lie superalgebra: $[a,\beta]=\alpha$, $[\beta,\beta]=b$, which is noted by $(C^3+A)$ in Backhouse's  list are isomorphic by means of the isomorphism that follows: 
$$a=X_1, \ b=-2X_0, \ \alpha=-2Y_2, \ \beta=Y_1$$

 On the other hand, it  can be easily checked that $\gg_{4,2}^s$ is not nilpotent and therefore it does not admit a symplectic structure. Thus, the only superalgebra candidate to be symplectic is  $\gg_{4,1}^s$. Since a necessary and sufficient condition for $\gg_{4,1}^s$ to be a symplectic Lie superalgebra $(\gg_{4,1}^s, \omega)$ is the existance of an even invertible skew-supersymmetric superderivation $\delta$ of $(\gg_{4,1}^s, B)$ verifying: $\omega(X,Y)=B(\delta(X),Y)$ for all $X$, $Y$, we look for such $\delta$. First, we set the images of the generator vectors:
$$\delta(X_1)=a_0X_0+a_1X_1, \quad \delta(Y_1)=b_1Y_1+b_2Y_2$$
From the condition of being an even superderivation we get

$$\delta([Y_1,Y_1])=2[Y_1,\delta(Y_1)] \Longleftrightarrow \delta(X_0)=2b_1X_0$$
$$\delta([X_1,Y_1])=[\delta(X_1),Y_1]+[X_1,\delta(Y_1)] \Longleftrightarrow \delta(Y_2)=(a_1+b_1)Y_2$$
Next, from the condition to be skew-supersymmetrical we have 
$$B(\delta(X_0),X_1)=-B(X_0,\delta(X_1)) \Longrightarrow a_1=-2b_1$$
$$B(\delta(X_1),X_1)=-B(X_1,\delta(X_1))\Longrightarrow a_0=0.$$
Thus, any even invertible skew-supersymmetric superderivation $\delta$ of $(\gg_{4,1}^s, B)$ verifies:
$$\delta(X_1)=-2b_1X_1, \ \delta(Y_1)=b_1Y_1+b_2Y_2, \ \delta(X_0)=2b_1X_0, \ \delta(Y_2)=-b_1Y_2, \ b_1 \neq 0.$$
After defining $\omega(X,Y):=B(\delta(X),Y)$ for all $X,Y$ we obtain the expression for $\omega$ of the statement, which concludes the proof.
\end{proof}

\begin{prop} \label{prop3.2} Let $(\gg,B)$ be a complex   quadratic Lie superalgebra of dimension $6$ \cite{classificationLow}. We have: \begin{itemize} 

\item[(i)] If $\dim ({\mathfrak{g}}_{\bar 1})=2$, then there is no quadratic symplectic Lie superalgebra. 

\item[(ii)] If $\dim ({\mathfrak g}_{\bar 1})=4$, then there exists only one quadratic symplectic Lie superalgebra, up to i-isomorphism, named $\gg_{6,4}^s = {\mathfrak g}_{\bar 0} \oplus {\mathfrak{g}}_{\bar 1}$, where ${\mathfrak g}_{\bar 0} = span_{\CC}\{X_0,X_1\}$ and ${\mathfrak g}_{\bar 1} = span_{\CC}\{Y_1,Y_2,Y_3,Y_4\}$ such that the non-zero bilinear form $B(X_0,X_1)=B(Y_4,Y_1)=B(Y_3,Y_2)=1$ and the non-trivial Lie super-brackets:
$$\gg_{6,4}^s:\left\{\begin{array}{l}
[X_1,Y_1]=-[Y_1,X_1]=-Y_2 \\{}
[X_1,Y_3]=-[Y_3,X_1]=Y_4 \\{}
[Y_3,Y_1]=[Y_1,Y_3]=X_0
\end{array}\right.$$
Once fixed $B$ above,  any  symplectic   structure $\omega$  on $\gg_{6,4}^s$ is given by:
$$\left\{\begin{array}{l}
\omega(X_0,X_1)=b_1+c_3\\ \omega(Y_4,Y_1)=-b_1\\ \omega(Y_3,Y_2)=c_3\\
\omega(Y_1,Y_1)=b_4\\ \omega(Y_1,Y_3)=-b_2\\ \omega(Y_3,Y_3)=-c_2\\
\end{array}\right.$$
\noindent where $b_1, b_2, b_4, c_2, c_3 $  are parameters  in $ \CC,$  with $b_1,c_3 \neq 0$ and $b_1 \neq -c_3$.
\end{itemize} 
\end{prop} 

\begin{proof} (i) If $\dim ({\mathfrak{g}}_{\bar 1})=2$, then ${\mathfrak{g}}_{\bar 0}$ is i-isomorphic to the diamond Lie algebra  $\gg_4$, see \cite{classificationLow},  $\gg_4=span_{\CC} \{ X,P,Q,Z\}$ such that $B(X,Z)=B(P,Q)=1$ and $[X,P]=P$, $[X,Q]=-Q$, $[P,Q]=Z$. It can be checked that $\gg_4$ is a solvable non-nilpotent Lie algebra and therefore, there is no indecomposable quadratic symplectic Lie superalgebra. 

\

(ii) If $\dim ({\mathfrak{g}}_{\bar 1})=4$, according to  \cite[Subsetion 2.3.2., page 132]{classificationLow} , there are, up to i-isomorphism, the following quadratic Lie superalgebras: $\gg_{6,4}^s$ (described in the statement of the proposition) together with $\gg_{6,5}^s$, $\gg_{6,6}^s(\lambda)$ and $\gg_{6,7}^s$. With respect to the basis $\{X_0, Y_0\}$ for the even part and $\{X_1,X_2,Y_1,Y_2\}$ for the odd one, they can be  expressed by :
 
$$\begin{array}{ll}
 \gg_{6,5}^s: & [Y_0,X_2] = X_2, [Y_0,Y_1] = X_1, [Y_0,Y_2]=-Y_2,  [Y_1,Y_1]=[X_2,Y_2]=X_0 \\ 
 
\\
 
 \gg_{6,6}^s(\lambda): & [Y_0,X_1]=X_1, [Y_0,X_2]=\lambda X_2, [Y_0,Y_1]=-Y_1, [Y_0,Y_2]=-\lambda Y_2,  \\{} & [X_1,Y_1]=X_0, [X_2,Y_2]=\lambda X_0\\
 
\\

  \gg_{6,7}^s: & [Y_0,X_1]=X_1, [Y_0,X_2]= X_2+X_1, [Y_0,Y_1]=-Y_1-Y_2, [Y_0,Y_2]=- Y_1,  \\{} & [X_1,Y_1]= [X_2,Y_1]=[X_2,Y_2]= X_0
 \end{array}$$
 
It can be easily checked that none of the solvable quadratic superalgebras $\gg_{6,5}^s$, $\gg_{6,6}^s(\lambda)$ and $\gg_{6,7}^s$ is nilpotent and therefore they do not admit a symplectic structure.  Thus, only $\gg^s_{6,4}$ is a candidate for being a symplectic Lie superalgebra.   Accordingly, we consider the fixed $B$ to be the same as in \cite{classificationLow}.  Since a necessary and sufficient condition for $\gg_{6,4}^s$ to be a symplectic Lie superalgebra $(\gg_{6,4}^s, \omega)$ is the existence of an even invertible skew-supersymmetric superderivation $\delta$ of $(\gg_{6,4}^s, B)$ verifying: $\omega(X,Y) = B(\delta(X),Y)$, we look for such $\delta$, analogously as it was done in the previous proposition.

First, we set the images of the generator vectors:
$\delta(X_1)=a_0X_0+a_1X_1,$ and 
$$\delta(Y_1)=b_1Y_1+b_2Y_2+b_3Y_3+b_4Y_4, \hspace{0.4cm} \delta(Y_3)=c_1Y_1+c_2Y_2+c_3Y_3+c_4Y_4.$$
From the condition of being an even superderivation we get
$$\delta(X_0)=(b_1+c_3)X_0, \  \delta(Y_2)=(a_1+b_1)Y_2, \ \delta(Y_4)=(a_1+c_3)Y_4, \ b_3=c_1=0.$$
Next, from the condition to be skew-supersymmetrical we have 
$$ a_1=-b_1-c_3, \quad a_0=0, \quad c_4=-b_2.$$
Thus, any even invertible skew-supersymmetric superderivation $\delta$ of $(\gg_{6,4}^s, B)$ is defined by the following matrix:
$$\left(  \begin{array}{cccccc}
b_1+c_3 & 0 &0 &0&0&0 \\
0& -b_1-c_3&0 &0&0&0 \\
0&0&b_1&0 &0&0 \\
0&0&b_2&-c_3 &c_2&0 \\
0&0&0&0 &c_3&0 \\
0&0&b_4&0 &-b_2&-b_1 \\
\end{array} \right)$$
with $b_1,  c_3 \neq 0, \ b_1 \neq -c_3$.

Then, after defining $\omega(X,Y):=B(\delta(X),Y)$ for all $X$, $Y$, we obtain the expression for $\omega$ of the statement, which concludes the proof. 
\end{proof}

\begin{rem} Let us remark that the symplectic quadratic Lie superalgebras $\gg_{4,1}^s$ and $\gg_{6,4}^s$ were mentioned  in \cite{JPAA(2009-ELS)} as examples of quadratic Lie superalgebras with a reductive even part and the action of the even part on  the odd part is not  completely reducible.
\end{rem}

\subsection{Double extensions of quadratic symplectic  Lie superalgebras}

Next,  we start with  double extensions of quadratic symplectic Lie superalgebras in order to obtain the conditions under which we get another quadratic symplectic Lie superalgebra as a result of the double extension. For that purpose we start with double extensions  by one-dimensional Lie algebras.  

\begin{lem}  \label{thm3}
Let $(\aa, B, \omega)$ be a finite-dimensional quadratic symplectic Lie superalgebra. Let $(\gg, \widetilde{B})$ be a double extension of the quadratic Lie superalgebra $(\aa, B)$ by the $1$-dimensional Lie algebra $(\KK e)_{\bar 0}$ by means of $\psi : \KK e \to {\rm Der}_a (\aa, B),$ defined by $\psi(e):=D,$ and $D$ an even skew-supersymmetric superderivation of $(\aa, B)$. If $(\gg, \widetilde{B})$ admits a symplectic structure  then $D$ is nilpotent.  
\end{lem} 

\begin{proof}
Thanks to Corollary \ref{cor} $\aa$ is a nilpotent Lie superalgebra, i.e.  $\aa_{\bar 0}$ is a nilpotent Lie algebra and the adjoint representation of $\aa_{\bar 0}$ over $\aa_{\bar 1}$ is nilpotent: for every $X \in \aa_{\bar 0}$, $\rho_{\aa}(X) \in {\mathfrak g}{\mathfrak l}(\aa_{\bar 1})$ is nilpotent.   Applying Theorem \ref{thm_Doubleextension} in our case we have the double extension $\gg = \KK e \oplus \aa \oplus \KK e^*$ with $\gg_{\bar 0} = \KK e \oplus \aa_{\bar 0} \oplus \KK e^*$ and $\gg_{\bar 1} = \aa_{\bar 1}$. By construction it can be seen the following conditions
$$\begin{array}{lll}
[e,X]&=&D(X) \\[1mm]
[X,Y]&=&[X,Y]_{\aa}+ \varphi(X,Y) \\[1mm]
[e^*,X]&=&[e,e^*]=0
\end{array}$$
for $X,Y \in \aa.$
Moreover, the fact that $[e,X]=D(X)$ implies that  the action of $\KK e$ over $\aa_{\bar 1} = \gg_{\bar 1}$ is nilpotent if and only if  the linear map $D$ is nilpotent. Since $\gg$ admits  a symplectic structure, $\gg$ is nilpotent and in particular the action of $\KK e$ over $\gg_{\bar 1}$ must be nilpotent which concludes the proof of the statement.
\end{proof}

In what follows we are going to apply the aforementioned double extensions to the quadratic symplectic  Lie superalgebras $\gg_{4,1}^s$ and $\gg_{6,4}^s$. For both of them we consider  the double extensions obtained by all the possible even nilpotent skew-supersymmetric superderivations. It can be checked that  all Lie superalgebras  double extensions obtained are symplectic. Therefore, a straightforward computation leads to the following results. 

\begin{prop} \label{prop3.3}
Let $(\gg, \widetilde{B})$ be a double extension of the   quadratic Lie superalgebra $(\gg_{4,1}^s, B)$ by the $1$-dimensional Lie algebra $(\CC e)_{\bar 0}$ by means of $\psi : \CC e \to {\rm Der}_a (\gg_{4,1}^s, B),$ defined by $\psi(e)=D,$ and $D$ an even nilpotent skew-supersymmetric superderivation of $(\gg_{4,1}^s, B)$.  Assume that the  quadratic Lie superalgebra $(\gg_{4,1}^s, B)$ is endowed with a symplectic structure presented in Proposition \ref{prop3.1}. Then, $(\gg, \widetilde{B},\widetilde{\omega})$ is a quadratic symplectic Lie superalgebra  whose structure  is i-isomorphism to one Lie superalgebra of the following family: 
$$\begin{array}{l}
[e,Y_1]=-[Y_1,e]=b_2 Y_2, \\{}
[X_1,Y_1]=-[Y_1,X_1]=-2Y_2 \\{}
[Y_1,Y_1]=-2X_0+b_2e^*
\end{array}$$
where $b_2$ is a non-zero parameter in $\CC,$ with respect  to the basis $X_0,X_1,e,e^*$ (even) and $Y_1, Y_2$ (odd).  Moreover, the symplectic structure can be $\widetilde{\omega}(X,Y):=\widetilde{B}(\delta(X),Y)$ with $\delta$ an even invertible skew-supersymmetric superderivation of $(\gg, \widetilde{B})$ defined by the  $(6 \times 6)$-matrix 
 $$2E^{1,1} - 2E^{2,2} - 2E^{3,3} + 2E^{4,4} + E^{5,5} - E^{6,6}.$$ 
\end{prop}

\begin{rem}
Note that the quadratic symplectic Lie superalgebra $(\gg, \tilde{B},\tilde{\omega})$   in Proposition \ref{prop3.3}, by \cite[Theorem 4.9]{simplecticLSA},  is  also a quadratic symplectic double extension of $(\gg_{4,1}^s,B,\delta_{4,1})$, where   $$\mbox{Mat}_{(4 \times 4)} (\delta_{4,1}, \{X_0,X_1, Y_1,Y_2\}) = -2E^{1,1}-2E^{2,2}+2E^{3,3}+E^{4,4}.$$ 
Note also that any $D$, even nilpotent skew-supersymmetric superderivation of $(\gg_{4,1}^s, B)$ is defined by the $(4 \times 4)$-matrix $E^{4,3}$. 
\end{rem}

\begin{prop}\label{prop 3.4}
Let $(\gg, \widetilde{B})$ be a double extension of the quadratic Lie superalgebra $(\gg_{6,4}^s, B)$ by the $1$-dimensional Lie algebra $(\CC e)_{\bar 0}$ by means of $\psi : \CC e \to {\rm Der}_a (\gg_{6,4}^s, B),$ defined by $\psi(e) := D,$ and $D$ an even nilpotent skew-supersymmetric superderivation of $(\gg_{6,4}^s, B)$. Assume that the  quadratic Lie superalgebra $(\gg_{6,4}^s, B)$ is endowed with a symplectic structure presented in Proposition \ref{prop3.2}. Then, $(\gg, \widetilde{B},\widetilde{\omega})$ is a quadratic symplectic Lie superalgebra whose law is i-isomorphism to one superalgebra of the following family: 

$$\begin{array}{l}
[e,Y_1]=-[Y_1,e]=b_2 Y_2+b_4Y_4  \\{}
[e,Y_3]=-[Y_3,e]=c_2 Y_2-b_2Y_4 \\{}
[X_1,Y_1]=-[Y_1,X_1]=-Y_2 \\{}
[X_1,Y_3]=-[Y_3,X_1]=Y_4 \\{}
[Y_1,Y_3]=[Y_3,Y_1]=X_0-b_2e^{*} \\{}
[Y_1,Y_1]=b_4e^*\\{}
[Y_3,Y_3]=-c_2e^*,
\end{array}$$
where $b_2,b_4,c_2$ are parameters in $\CC$ such that $(b_2,b_4,c_2)\neq (0,0,0)$, with respect to the basis $\{X_0,X_1,e,e^*,Y_1,Y_2,Y_3,Y_4\}$, being $\{X_0,X_1,e,e^*\}$ and $\{Y_1,Y_2,Y_3,Y_4\}$ bases of $\gg_{\bar 0}$ and $\gg_{\bar 1}$, respectively. Moreover, the symplectic structure can be given by $\widetilde{\omega}(X,Y):=\widetilde{B}(\delta(X),Y)$ with $\delta$ an even invertible skew-supersymmetric superderivation of $(\gg, \widetilde{B})$ defined by the $(8 \times 8)$-matrix $$2E^{1,1}-2E^{2,2}-2E^{3,3}+2E^{4,4}+E^{5,5}-E^{6,6}+E^{7,7}-E^{8,8}.$$ 
\end{prop}

\begin{rem}
We can check that the quadratic symplectic Lie superalgebra $(\gg, \tilde{B},\tilde{\omega})$ in Proposition \ref{prop 3.4}, by \cite[Theorem 4.9]{simplecticLSA}, is  also a quadratic symplectic double extension of $(\gg_{6,4}^s,B,\delta_{6,4})$, where  $$\mbox{Mat}_{(6 \times 6)}(\delta_{6,4},\{X_0,X_1, Y_1,Y_2,Y_3,Y_4\}) = -2E^{1,1}-2E^{2,2}+2E^{3,3}+E^{4,4}-E^{5,5}+E^{6,6}.$$ 
Note also that any $D$,  even nilpotent skew-supersymmetric superderivation of $(\gg_{6,4}^s, B)$ is defined by the  $(6 \times 6)$-matrix  
$$  b_2E^{4,3}+c_2E^{4,5}+b_4E^{6,3}-b_2E^{6,5}.$$

\end{rem}

\subsection{Generalized double extensions of quadratic symplectic  Lie superalgebras}

Recall now the notion of generalized double extension (see \cite{2}). Thus, suppose $\displaystyle (\gg = \gg_{\bar 0} \oplus \gg_{\bar 1},B)$ a quadratic Lie superalgebra and $\displaystyle D$ an odd skew-supersymmetric superderivation of $\displaystyle (\gg,B)$, $\displaystyle X_0$ a nonzero element of
$\displaystyle \gg_{{\bar 0}}$ such that $D(X_0) = 0, \;
B(X_0,X_0) = 0, \;
D^2 = \frac{1}{2}[X_0,.]_{\gg}$. 

Then the generalized double extension $\displaystyle (\widetilde{\gg}=\KK e \oplus \gg \oplus \KK e^*,\widetilde{B}) $  of $\displaystyle (\gg,B)$
by the $1$-dimensional Lie superalgebra $\displaystyle
(\KK e)_{\bar 1}$ (by means of the odd skew-supersymmetric superderivation $D$ and $\displaystyle X_0$) is a quadratic Lie superalgebra  with  $\dim({\widetilde{\gg}}_{\bar 1}) = \dim(\gg_{\bar 1}) + 2  $.  The brackets on $\widetilde{\gg}$ are defined by $\displaystyle
[e,e] = X_0, \; [e^*,\widetilde{\gg}] = \{0\}$, $\displaystyle
[e,X] =  D(X)-B(X,X_0) e^*$, $[X,Y] = [X,Y]_{\gg} - B(D(X),Y) e^*,$ for all $X,Y \in \gg,$ 
and the supersymmetric bilinear form $\displaystyle \widetilde{B}: \widetilde{\gg} \times \widetilde{\gg} \longrightarrow \KK$ is defined by $\displaystyle
\widetilde{B}\mid_{{\gg}\times {\gg}} =B,
 \widetilde{B}(e^*,e) = 1,
 \widetilde{B}(\gg,e)=  \widetilde{B}(\gg,e^*) = \{0\} $, $\widetilde{B}(e,e)= \widetilde{B}(e^*,e^*) = 0.$
The quadratic Lie superalgebra $\displaystyle (\widetilde{\gg}=\KK e \oplus {\gg} \oplus \KK e^*,\widetilde{B})$ is called an {\it generalized double extension of} $(\gg, B)$ by the $1$-dimensional Lie superalgebra $\displaystyle
(\KK e)_{\bar 1}$ (by means of the odd skew-supersymmetric superderivation $D$ and $\displaystyle X_0$).

\begin{thm}
Suppose   $\displaystyle (\gg = \gg_{\bar 0} \oplus \gg_{\bar 1},B,\omega)$  is a quadratic symplectic Lie superalgebra, then its generalized double extension $\displaystyle (\widetilde{\gg} = \KK e \oplus \gg \oplus \KK e^*,\widetilde{B})$ is always a nilpotent quadratic Lie superalgebra.
\end{thm}

\begin{proof}
Let us note that $\widetilde{\gg}$ is nilpotent if and only if $\widetilde{\gg}_{\bar 0}$ is nilpotent and, for all $X \in \widetilde{\gg}_{\bar 0}$,  $\widetilde{ad}_X$ is nilpotent. 

Thanks to Corollary \ref{cor} $\gg$ is a nilpotent Lie superalgebra, i.e.  $\gg_{\bar 0}$ is a nilpotent Lie algebra and the adjoint representation of $\gg_{\bar 0}$ over $\gg_{\bar 1}$ is nilpotent:  for every $X \in \gg_{\bar 0}$, $\rho_{\gg}(X) \in {\mathfrak g}{\mathfrak l}(\gg_{\bar 1})$ is nilpotent, i.e. $ad_X$ is nilpotent. After applying the generalized double extension $\widetilde{\gg} = \KK e \oplus \gg \oplus \KK e^*$ with $\widetilde{\gg}_{\bar 0} = \gg_{\bar 0}$ and $\widetilde{\gg}_{\bar 1} = \KK e \oplus \gg_{\bar 1} \oplus \KK e^*$ we have
$$\begin{array}{lll}
[e,e]&=& X_0\\[1mm]
[e,X]&=&D(X)-B(X,X_0)e^* \\[1mm]
[X,Y]&=&[X,Y]_{\gg}-B(D(X),Y) e^*\\[1mm]
\end{array}$$
for $X,Y \in \gg.$ Since $\widetilde{\gg}_{\bar 0} ={\gg}_{\bar 0}$ we already have that $\widetilde{\gg}_{\bar 0}$ is nilpotent, therefore only rest to check that $\widetilde{ad}_X$ is nilpotent for all $X \in \widetilde{\gg}_{\bar 0}$.

Let $X \in \gg_{\bar 0} = \widetilde{\gg}_{\bar 0}$, since $e^* \in \mathfrak{z}(\widetilde{\gg})$, we get for all homogeneous element  $ Y \in \gg$ 
\begin{align}
(\widetilde{ad}_ X)^2(Y) &= [X,[X,Y]] \nonumber \\ 
&= [X,[X,Y]_{\gg}-B(D(X),Y) e^*] \nonumber\\
&= [X,[X,Y]_{\gg}]  \nonumber\\
&=[X,[X,Y]_{\gg}]_{\gg}-B(D(X),[X,Y]_{\gg}) e^* \nonumber\\ 
&=(ad_{X})^2(Y)+B(X,D([X,Y]_{\gg})) e^* \nonumber\\
&=(ad_{X})^2(Y)+B( D([X,Y]_{\gg}),X) e^* \nonumber\\
&= (ad_{X})^2(Y) + B(D(ad_{X}(Y)), X )e^* \nonumber
\end{align}
By iterating this process we get
\begin{align}
(\widetilde{ad}_X)^n(Y) = (ad_X)^n(Y) + B\bigl(D\bigl((ad_X)^{n-1}(Y)\bigr), X\bigr)e^* \nonumber .
\end{align}

On account of $ad_X$ is nilpotent we have that $\widetilde{ad}_X$ acting on $\gg$ is also nilpotent. Since $e^* \in \mathfrak{z}(\widetilde{\mathfrak g})$, only remains to see the action of the adjoint operator $\widetilde{ad}_X$ over $e$. From $[e,X]=D(X)-B(X,X_0)e^*$ we get for any $X$ even basis vector
$$\widetilde{ad}_X(e)=-D(X)+B(X,X_0)e^*$$
By construction we have $D^2  = \frac{1}{2}[X_0,\cdot]_{\mathfrak g}=\frac{1}{2}ad_{X_0}$ with $ad_{X_0}$ nilpotent. Consequently  $D^2$ and  $D$ are nilpotent. It can be then, easily checked that the action of the adjoint operator $\widetilde{ad}_X$ over $e$ is also nilpotent,  in fact
\begin{align}
(\widetilde{ad}_X)^n(e) = (ad_X)^{n-1}(D(X)) + B(D(X), (ad_X)^{n-2}(D(X)) \nonumber 
\end{align}

\noindent which concludes the proof.
\end{proof}

In what follows, we study the  generalized  double extensions of the quadratic symplectic Lie superalgebra $\gg_{4,1}^{s}$ described in Proposition \ref{prop3.1}, i.e. the only one, up to i-isomorphism verifying $\dim(\gg_{\bar 0})=\dim(\gg_{\bar 1})=2.$ In particular, we show the connection between  $\gg_{4,1}^{s}$ and $\gg_{6,4}^s$ which is the only quadratic symplectic Lie superalgebra, up to i-isomorphism, verifying $\dim(\gg_{\bar 0})=2, \ \dim(\gg_{\bar 1})=4$, see Proposition \ref{prop3.2}.  

\begin{prop}
If $({\mathfrak g}, \widetilde{B})$ is a complex  quadratic Lie superalgebra which is a generalized double extension of the quadratic symplectic Lie superalgebra $(\gg_{4,1}^s, B, \omega)$, then $\gg$ is i-isomorphism to one superalgebra of the following family of quadratic Lie superalgebras expressed by the  non-null bracket products that follow: $$\begin{array}{l}
[e,e]=X_0 \\{}
[e,X_1]=-[X_1,e]=a_1Y_1+a_2Y_2-e^* \\{}
[e,Y_1]=[Y_1,e]=-a_2 X_0 \\{}
[e,Y_2]=[Y_2,e]=a_1 X_0 \\{}
[X_1,Y_1]=-[Y_1,X_1]=-2Y_2-a_2e^* \\{}
[X_1,Y_2]=-[Y_2,X_1]=a_1e^* \\{}
[Y_1,Y_1]=-2X_0,
\end{array}$$
where $a_1,a_2$ are parameters in $\CC$, with respect to the basis $\{X_0,X_1,Y_1,Y_2,e,e^*\}$, being $\{X_0,X_1\}$ a basis of $\gg_{\bar 0}$ and $\{ Y_1,Y_2,e,e^{*}\}$ a basis of $\gg_{\bar 1}$, and having $\widetilde{B}(X_0 ,X_1) = \widetilde{B}(Y_2,Y_1) = \widetilde{B}(e^*,e) = 1$. If $a_1=0$  all the quadratic superalgebras of the above family admit a symplectic structure which can be defined by $\widetilde{\omega}(X,Y):=\widetilde{B}(\delta(X),Y)$ with $\delta$ an even invertible skew-supersymmetric superderivation of $({\mathfrak g}, \widetilde{B})$ defined by the $(6\times 6)$-matrix  
$$  2cE^{1,1}-2cE^{2,2}+cE^{3,3}-cE^{4,4}+cE^{5,5}-cE^{6,6}.$$
where $c \neq 0$. Moreover, if $a_1=a_2=0$, then $\gg$ is isomorphic to $\gg_{6,4}^s$. 
\end{prop}
\begin{proof}
Starting from $(\gg_{4,1}^s,B)$ we compute all the odd skew-supersymmetric superderivations $D$, obtaining the following expression for them:
$$D(X_1)=a_1Y_1+a_2Y_2, \quad D(X_0)=0, \quad D(Y_1)=-a_2X_0, \quad D(Y_2)=a_1X_0$$
Let us note that all them verify the conditions required in  order to be able to apply the generalized double extension, i.e. $D(X_0) = 0, \;
B(X_0,X_0) = 0, \;
D^2  = \frac{1}{2}[X_0,\cdot]_{\mathfrak g}$, in particular $D^2$ is the null derivation which is the same as $\frac{1}{2}ad_{X_{0}}$ since $X_0$ is a central element of the superalgebra. Consequently, we get the following expression for the generalized double extension quadratic Lie superalgebra  with respect to the basis $\{X_0,X_1,Y_1,Y_2,e,e^*\}$, where $\{X_0,X_1\}$ and $\{ Y_1,Y_2,e,e^{*}\}$ are the bases of $\gg_{\bar0}$ and $\gg_{\bar 1}$, respectively:
$$\begin{array}{l}
[e,e]=X_0 \\{}
[e,X_1]=-[X_1,e]=a_1Y_1+a_2Y_2-e^{*}  \\{}
[e,Y_1]=[Y_1,e]=-a_2 X_0 \\{}
[e,Y_2]=[Y_2,e]=a_1 X_0 \\{}
[X_1,Y_1]=-[Y_1,X_1]=-2Y_2-a_2e^{*} \\{}
[X_1,Y_2]=-[Y_2,X_1]=a_1e^{*} \\{}
[Y_1,Y_1]=-2X_0
 \end{array} $$
 and $ \widetilde{B}(X_0 ,X_1)=\widetilde{B}(Y_2,Y_1) = \widetilde{B}(e^*,e) = 1$. For $a_1=0$ and looking for a diagonal even invertible derivation, we start setting for the generators $$\delta(X_1)=b_1X_1, \quad \delta(e)=ce, \quad b_1, c \neq 0.$$
 From the condition of being a derivation over $[e,e]$ we get $\delta(X_0)=2cX_0$ and then over $[Y_1,Y_1]$ we assert that $\delta(Y_1) = cY_1$. Now by applying the condition of derivation over $[e,X_1]$ we obtain $\delta(e^*)=(c+b_1)e^*$. Next, from the condition of derivation over $[X_1,Y_1]$ it is obtained $\delta(Y_2) = (c+b_1)Y_2$. Finally, the condition of being skew-supersymmetric leads to $b_1=-2c$ obtaining then the derivation of the statement.
 
 \
 
Furthermore, it can be checked that if we apply the i-isomorphism defined by $X'_0=X_0, \ X'_1=X_1$ together with
$$ Y'_1=Y_1-Y_3, \ Y'_2=\frac{1}{2}(Y_2+Y_4), Y'_3=\frac{1}{\sqrt{2}}(Y_1+Y_3), \ Y'_4=\frac{1}{\sqrt{2}}(-Y_2+Y_4)$$
to the Lie superalgebra  $$\gg_{6,4}^s:  [X_1,Y_1]=-[Y_1,X_1]=-Y_2,  [X_1,Y_3]=-[Y_3,X_1]=Y_4,  [Y_3,Y_1]=[Y_1,Y_3]=X_0$$
we get 
$$\begin{array}{ll}
[Y'_3,Y'_3]=X'_0, &
[Y'_3,X'_1]=-[X'_1,Y'_3]=-Y'_4,  \\{}
[X'_1,Y'_1]=-[Y'_1,X'_1]=-2Y'_2, &
[Y'_1,Y'_1]=-2X'_0,
\end{array} $$
and by renaming $Y'_3=e$, $Y'_4=e^*$, $Y'_1=Y_1$, $Y'_2=Y_2$, $X'_0=X_0$ and $X'_1=X_1$ we  obtain the superalgebra of the family of the statement with $a_1=a_2=0$, which concludes the proof.
\end{proof}

\section{Inductive description of quadratic symplectic Lie superalgebras of filiform type}

The purpose of this section is to provide an inductive description of quadratic symplectic
Lie superalgebras of filiform type. The quadratic symplectic Lie superalgebra
$\gg_{4,1}^s$ described in \ref{prop3.1} and the family of quadratic symplectic Lie superalgebras described in \ref{prop3.3} are all filiform.

\begin{lem} \label{lem1}
Let $(\gg,B,\omega)$ be a quadratic symplectic Lie superalgebra of filiform type  $\gg_{\bar 1}$, $$\gg_{\bar 1}=V_m \supset \dots \supset V_1
\supset V_0.$$ and $\delta$ the unique even invertible
skew-supersymmetric superderivation of $(\gg,B)$
such that $\omega (X,Y)= B(\delta(X),Y)$, for all $X,Y\in \gg$.
Then:
\begin{enumerate}
\item[(i)] $\omega (V_1, V_1)= \{0\}.$ \item[(ii)] $V_1= {\mathfrak z}(\gg) \cap \gg_{\bar 1}.$
\end{enumerate}
\end{lem}
\begin{proof}
(i) 
Suppose $V_i =span_{\mathbb{K}}\{e_1,\dots,e_i\}$, for $1 \leq i \leq m$. Since   $[\gg_{\bar 0},V_{2}]= V_1$, there exists $X \in \gg_{\bar 0}$ such that $[X,e_2]=\lambda e_1$ with $\lambda \neq 0$. Since $V_1 \subset {\mathfrak z}(\gg)  \cap {\gg}_{\bar 1}$  (see \cite[Lemma 2.12]{quadraticFLSA}),  $[\gg_{\bar 0},V_{1}]= \{0\}$ and $\omega$ is skew-suppersymmetric on $\gg$, we have
\begin{align*}
\omega(e_1,e_1)&=\lambda^{-1}\omega(\lambda e_1,e_1) = \lambda^{-1}\omega([X,e_2],e_1) =\lambda^{-1} \omega(e_1,[X,e_2]) \\
&= \lambda^{-1}\bigl(\omega(e_2,[e_1,X]) +\omega(X,[e_2,e_1])\bigr) = 0.
\end{align*}

(ii) It was proved in \cite[Lemma 2.12]{quadraticFLSA} that $V_1 \subset {\mathfrak z}(\gg)  \cap {\gg}_{\bar 1}$, and since  $\gg$ is of filiform type  it can be checked that ${\mathfrak z}(\gg)  \cap {\gg}_{\bar 1} \subset V_1$ which conludes the proof. In fact, suppose as before $V_i =span_{\mathbb{K}}\{e_1,\dots,e_i\}$, for $1 \leq i \leq m$. Since $[\gg_{\bar 0},V_i]= V_{i-1}$, for $2 \leq i \leq m$ there exist $X_i \in \gg_{\bar 0}$ such that $[X_i,e_i]=\lambda_i e_i + \sum_{j=1}^{i-1} \lambda_j e_j$ with $\lambda_i \neq 0$. Consequently,   ${\mathfrak z}(\gg)  \cap {\gg}_{\bar 1} \subset span_{\mathbb{K}}\{e_1\}=V_1$.  
\end{proof}

\begin{rem}\label{Remark_ideal}
If $I$ is a graded ideal of $\gg$, in general the orthogonal of $I$ with respect to $\omega$, denoted by $I^{\perp}$, is not a graded ideal, it just satisfies $[I^{\perp} ,I^{\perp}] \subseteq I^{\perp}$. But if $I \subseteq {\mathfrak z}(\gg)$ then $I^{\perp}$ is a graded ideal. Indeed, for $X \in \gg$ we get $\omega([I,X],I^{\perp})=0$ and $\omega([I^{\perp},I],X) =0$. Since $\omega$ is scalar $2$-cocycle on $\gg$ we conclude that $\omega([X,I^{\perp}],I) =0$, that is, 
$[\gg ,I^{\perp}  ] \subseteq I^{\perp}$.
\end{rem}

Assume that  $\displaystyle ({\gg},B,\omega)$ is a quadratic
symplectic Lie superalgebra of   filiform type  with $dim({\gg}_{\bar 1})=m > 0$ and such that ${\gg}_{\bar 1}$ is a  \textit{filiform ${\gg}_{\bar 0}$-module}  with a flag  $${\gg}_{\bar 1}=V_m \supset \dots \supset V_1
\supset V_0.$$ Set that  $V_i :=span_{\mathbb{K}}\{e_1,\dots,e_i\}$, for $1 \leq i \leq m$.  Denote $I := V_1 = \mathbb{K}e_1$.\\

Let us recall now the concept of elementary odd double extensions for quadratic Lie algebras and superalgebras \cite{quadraticFLSA}. Thus, let $\displaystyle (\gg,B)$ be a quadratic Lie algebra with non-null center.
Let $\displaystyle X_0 \in {\mathfrak z}(\gg) \backslash \{0\}$ such that $B(X_0,X_0) = 0$.
We denote by $\displaystyle \KK e$ the $1$-dimension $\mathbb{Z}_2$-graded vector space such that $\displaystyle \KK e = {(\KK e)}_{\bar 1}$ and 
$\displaystyle \KK e^*$ its dual vector space.

Then the $\mathbb{Z}_2$-graded vector space $\displaystyle \widetilde{\gg} = {\widetilde{\gg}}_{\bar 0} \oplus {\widetilde{\gg}}_{\bar 1}$, where ${\widetilde{\gg}}_{\bar 0} = \gg$ and ${\widetilde{\gg}}_{\bar 1} = \KK e \oplus \KK e^*$, with the even
skew-symmetric bilinear map $\displaystyle [\cdot,\cdot] : \widetilde{\gg} \times
\widetilde{\gg} \longrightarrow \widetilde{\gg}$ defined by:
$\displaystyle
[e,e] = X_0, \; [e^*,\widetilde{\gg}] = \{0\}$, $\displaystyle
[e,X] =  -B(X,X_0)e^*$, 
$[X,Y] = [X,Y]_{\gg},$ for all $X,Y \in \gg,$ 
and the supersymmetric
bilinear form $\displaystyle \widetilde{B}: \widetilde{\gg}\times \widetilde{\gg} \longrightarrow \KK$ defined by: 
$\displaystyle
\widetilde{B}\mid_{{\gg}\times {\gg}} = B,
 \widetilde{B}(e^*,e) = 1,
 \widetilde{B}(\gg,e) = \widetilde{B}(\gg,e^*) = \{0\} $, $\widetilde{B}(e,e)=
\widetilde{B}(e^*,e^*) = 0$
is a quadratic Lie superalgebra of   filiform type  with $dim({\widetilde{\gg}}_{\bar 1})= 2$, such that ${\widetilde{\gg}}_{\bar 1}$ is a filiform ${\widetilde{\gg}}_{\bar 0}$-module with the flag given by: $${\widetilde{\gg}}_{\bar 1}={\widetilde{V}}_2 \supset  {\widetilde{V}}_1
\supset {\widetilde{V}}_0,$$ where ${\widetilde{V}}_2=\KK e \oplus \KK e^*$, 
${\widetilde{V}}_1 = \KK e^*$,
and ${\widetilde{V}}_0=\{0\}$. 
The quadratic Lie superalgebra $\displaystyle (\widetilde{\gg}=\KK e \oplus {\gg} \oplus \KK e^\ast,\widetilde{B})$ is called the \emph{elementary
odd double extension} of $\displaystyle (\gg,B)$ by the $1$-dimensional Lie superalgebra $\displaystyle
(\KK e)_{\bar 1}$ (by means of $\displaystyle X_0$).


\begin{thm} \label{elementary odd double extion}
Assume that $\displaystyle (\gg,B,\omega)$ is a quadratic symplectic Lie algebra. Let  $\displaystyle X_0 \in {\mathfrak z}(\gg) \backslash \{0\}$  such that $B(X_0,X_0) = 0$. Let $\delta$ be the unique invertible skew-symmetric derivation of $\displaystyle
(\gg,B)$ such that $\displaystyle \omega (X,Y) = B(\delta(X),Y)$ for all $X,Y \in \gg$. 
Let $\displaystyle (\widetilde{\gg} = \KK e \oplus \gg \oplus \KK e^*,\widetilde{B})$ be the elementary odd double extension of $\displaystyle (\gg,B)$ by the $1$-dimensional Lie superalgebra $\displaystyle
(\KK e)_{\bar 1}$ (by means of $\displaystyle X_0$). If there exists $\displaystyle \beta \in \KK \backslash \{0\}$ such that
\begin{eqnarray}\label{Eq5.1}
&&  \delta(X_0) =\beta X_0,
\label{eq1}
\end{eqnarray}
then the linear endomorphism  $\widetilde{\delta} : \widetilde{\gg} \to \widetilde{\gg}$ defined by
\begin{align*}
 \widetilde{\delta}(e^*) &:= -\frac{\beta}{2} e^*,\\
\widetilde{\delta}(e) &:= \frac{\beta}{2} e,\\
\widetilde{\delta}(X) &:= \delta(X), \; \mbox{ for} \: X \in \gg,
\end{align*}
is an even invertible skew-supersymmetric superderivation of
$\displaystyle (\widetilde{\gg},\widetilde{B}) $. Then, by defining $\displaystyle  \widetilde{\omega} (X,Y):=
\widetilde{B}(\widetilde{\delta}(X),Y)$ for $X,Y\in \widetilde{\gg}$ we obtain  $\displaystyle (\widetilde{\gg}, \widetilde{B}, \widetilde{\omega})$ which is a quadratic
symplectic Lie superalgebra of   filiform type, i.e. ${\widetilde{\gg}}_{\bar 1}$ is a  \textit{filiform ${\widetilde{\gg}}_{\bar 0}$-module} with the flag defined by $${\widetilde{\gg}}_{\bar 1} = {\widetilde V}_2 \supset {\widetilde V}_1
\supset {\widetilde V}_0,$$ where
$ {\widetilde V}_2 := \KK e \oplus \KK e^*$, 
$ {\widetilde V}_1 := \KK e^*$ and 
$ {\widetilde V}_0 := \{0\}$. 

In this case, the quadratic symplectic Lie superalgebra $\displaystyle (\widetilde{\gg},\widetilde{B},\widetilde{\omega})$ is called the elementary odd double extension of  the quadratic
symplectic Lie algebra $\displaystyle (\gg,B,\omega)$ by the $1$-dimensional Lie superalgebra $\displaystyle
(\KK e)_{\bar 1}$ (by means of $\displaystyle X_0$ and $\beta$). \label{teo:asdfg234}
\end{thm}

\begin{proof} We have just to see that $\displaystyle \widetilde{\delta}$ is an even invertible skew-supersymmetric superderivation of $\displaystyle (\gg,\widetilde{B})$. Since
$\delta$ is invertible it is clear that $\displaystyle \widetilde{\delta}$ is also invertible. We start by showing that $\displaystyle \widetilde{\delta}$ is an even superderivation of $\displaystyle \widetilde{\gg}$. We have to prove in particular that $\displaystyle
\widetilde{\delta}([X,Y]) = [\widetilde{\delta}(X),Y] + [X,\widetilde{\delta}(Y)]$ for all $X,Y \in \gg$.
Computing the terms of the expression above, we obtain
\begin{eqnarray*}
\displaystyle && \widetilde{\delta}([X,Y]) \;=\; \delta([X,Y]_{\gg}),\\
\hspace*{0cm} && [\widetilde{\delta}(X),Y] \;=\; [\delta (X),Y]_{\gg},\\
\hspace*{0cm}&&[X,\widetilde{\delta}(Y)] \;=\; [X,\delta(Y)]_{\gg}.
 \end{eqnarray*}
As $\delta$ is skew symmetric,
we conclude that the equation above is equivalent to Eq. \eqref{Eq5.1}. 
Moreover, we have to show that
\begin{eqnarray}
\displaystyle \widetilde{\delta}([e,X]) &=& [\widetilde{\delta}(e),X] + [e,\widetilde{\delta}(X)], \hspace*{0.2cm} \mbox{for all } X \in \gg. \label{eq2}
\end{eqnarray}
Once more, calculating the terms, we obtain
\begin{eqnarray*}
\displaystyle && \widetilde{\delta}([e,X]) \;=\;  B(X,X_0)\frac{\beta}{2} e^*,\\
\hspace*{0cm} && [\widetilde{\delta}(e),X] \;=\;  - \frac{\beta}{2} B(X,X_0) e^*,\\
\hspace*{0cm} && [e,\widetilde{\delta}(X)] \;=\; -B(\delta(X),X_0)e^*.
\end{eqnarray*}
Then \eqref{eq2} is equivalent to  $$\beta B(X,X_0) + B(\delta(X),X_0) = 0.$$ 
As $\delta$ is skew-symmetric, doing some
calculations we conclude that the last assertion is expression \eqref{eq1}. Further, we easily verify that $\displaystyle
\widetilde{\delta}([e,e]) = [\widetilde{\delta}(e),e] + [e,\widetilde{\delta}(e)]$ is equivalent to \eqref{eq1}. Due to $\displaystyle
[e^*,\gg] = \{0\}$, it is immediate that $\displaystyle \widetilde{\delta}([e^*,X])=[\widetilde{\delta}(e^*),X] + [e^*,\widetilde{\delta}(X)]$, where $\displaystyle X=e$,
$\displaystyle X=e^*$, and $X \in \gg$. Finally, it can be easily checked that the even superderivation $\displaystyle \widetilde{\delta}$ is skew-supersymmetric since $\delta$ is a skew-symmetric superderivation of $\displaystyle (\gg,B)$ and $\displaystyle
\widetilde{B}\mid_{\gg \times \gg} = B,
\widetilde{B}(e^*,e) = 1,
\widetilde{B}(\gg,e) = \widetilde{B}(\gg,e^*) = \{0\}$, $\widetilde{B}(e,e) = \widetilde{B}(e^*,e^*) = 0$.
\end{proof}

\begin{prop} \label{prop4.1}
Let $(\gg,B,\omega)$ be a quadratic symplectic
Lie superalgebra of filiform type  and $\delta$ the unique even invertible skew-supersymmetric superderivation of $(\gg,B)$
such that $\omega (X,Y)= B(\delta(X),Y)$, for all $X,Y\in \gg$. Consider $D$ an even skew-super\-symmetric superderivation of $(\gg,B)$ and $(\widetilde{\gg} = \KK e \oplus \gg \oplus \KK e^*,\widetilde{B}) $ the double extension of $(\gg,B) $ by the $1$-dimensional Lie algebra $\KK e$ (by means of $D$). If
there exist $\alpha \in \KK \backslash \{0\}$ and $A_0 \in \gg_{\bar 0}$ such that
\begin{eqnarray}
\displaystyle  && [\delta,D]+\alpha D = {\rm ad}_{A_0},
\label{eq:nhju}
\end{eqnarray}
then the linear endomorphism  $\displaystyle \widetilde{\delta}: \widetilde{\gg}\longrightarrow \widetilde{\gg}$ defined by
\begin{eqnarray*}
\displaystyle && \widetilde{\delta}(e^\ast) \;=\;  \alpha e^\ast,\\
\hspace*{0cm}&& \widetilde{\delta}(e) \;=\; -\alpha e+A_0,\\
\hspace*{0cm}&&\widetilde{\delta}(X) \;=\;
\delta(X)-B(X,A_0)e^\ast,\hspace*{0.5cm}\forall \: { X \in \gg},
\end{eqnarray*}
is an even invertible skew-supersymmetric superderivation of
$\displaystyle(\widetilde{\gg},\widetilde{B})$. Consequently,
$\displaystyle (\widetilde{\gg},\widetilde{B},\widetilde{\omega}) $ is a
quadratic symplectic Lie superalgebra, where $\displaystyle
\widetilde{\omega}:\widetilde{\gg}\times \widetilde{\gg}\longrightarrow \KK$ is defined by
\begin{eqnarray*}
\displaystyle  \widetilde{\omega} (X,Y)&=&
\widetilde{B}(\widetilde{\delta}(X),Y),\hspace*{0.5cm}\forall \: { X,Y\in
\widetilde{\gg}}.
\end{eqnarray*}
In this case, the quadratic  symplectic Lie superalgebra
$\displaystyle (\widetilde{\gg},\widetilde{B},\widetilde{\omega}) $ is of   filiform type. It is also
called the quadratic symplectic double extension \index{quadratic
symplectic double extension!of a quadratic symplectic Lie
superalgebra} of $\displaystyle ({\gg},B,\omega) $ by the
one-dimensional Lie algebra $\displaystyle \KK e$ (by means of
$D$, $\displaystyle A_{0} $, and $\alpha$). \label{teo:asdfg}
\end{prop}

\begin{proof}
Thanks to \cite[Theorem 4.7]{simplecticLSA}, it remains to check that $\widetilde{\gg}$ is of filiform type. Using Theorem \ref{thm3} we deduce that $D$ is nilpotent, then from \cite[Proposition 5.1 (i)]{quadraticFLSA} we deduce that if $\gg$ is of filiform type
and the matrix associated with the linear map $D$ is nilpotent, then
the double extension Lie superalgebra $\widetilde{\gg}$ is nilpotent with a filiform
$\widetilde{\gg}_{\bar 0}$-module, which concludes the proof of the statement.
\end{proof}

\begin{exa}
For the  quadratic symplectic Lie superalgebra of filiform type  $(\gg_{4,1}^s, B, \omega)$, it has been shown in  Proposition \ref{prop3.3}  that all its double extensions  by the $1$-dimensional Lie algebra $(\mathbb{C}e)_{\bar 0}$ by means of $\psi : \mathbb{C}e \to {\rm Der}_a (\gg_{4,1}^s, B),$ defined by $\psi(e)=D,$ and $D$ an even nilpotent skew-supersymmetric superderivation of $(\gg_{4,1}^s, B)$, are  symplectic quadratic Lie superalgebras whose law is isomorphic to one superalgebra of the following family: 
$$\begin{array}{l}
[e,Y_1]=-[Y_1,e]=b_2Y_2 \\{}
[X_1,Y_1]=-[Y_1,X_1]=-2Y_2 \\{}
[Y_1,Y_1]=-2X_0+b_2e^*
\end{array}$$
where $b_2$ is a non-zero parameter in $\CC$, with respect to the basis $\{X_0,X_1,e,e^*,Y_1,Y_2\}$, where $\{X_0,X_1,e,e^*\}$ is a basis of its even part and $\{Y_1,Y_2\}$ is a basis of its odd part. It can be easily checked that this family is of filiform type with respect to the same flag as $\gg_{4,1}^s$:
$$ V_2 \supset V_1 \supset  V_0$$  where $V_2:=span_{\CC}\{Y_1,Y_2\}$, $V_1:=span_{\CC}\{Y_2\}$ and $V_0:=\{ 0\}$. Moreover, for any $\displaystyle \alpha \in \mathbb{C} \backslash \{0\}$ with $\alpha \neq 2b_1$ and $b_1\neq 0$ there exists 
$\displaystyle A_0 \in \gg_{\bar 0}$ such that $[\delta,D]+\alpha D =  \mbox{ad } (A_0)$. In particular $A_0 = \frac{(-2b_1+\alpha)b_2}{-2}X_1$ and the linear endomorphism  $\widetilde{\delta}: \widetilde{\gg}\to \widetilde{\gg}$ defined by
\begin{eqnarray*}
\displaystyle && \widetilde{\delta}(e^\ast) \;:=\; \alpha e^\ast,\\
\hspace*{0cm}&& \widetilde{\delta}(e) \;:=\; -\alpha e + \Bigl(b_1-\frac{\alpha}{2}\Bigr)b_2X_1, \\
\hspace*{0cm}&&\widetilde{\delta}(X_0) \;:=\;
2b_1X_0 + \Bigl(b_1-\frac{\alpha}{2}\Bigr)b_2 e^\ast, \\
 \hspace*{0cm}&&\widetilde{\delta}(X_1) \;:=\;
 -2b_1X_1,\\
 \hspace*{0cm}&&\widetilde{\delta}(Y_1) \;:=\;
 b_1Y_1+b_2Y_2,\\
 \hspace*{0cm}&&\widetilde{\delta}(Y_2) \;:=\;
 -b_1Y_2,
\end{eqnarray*}
is an even invertible skew-supersymmetric superderivation of
$\displaystyle(\widetilde{\gg},\widetilde{B})$.
\end{exa}

\begin{rem}
Let us note that the converse of Proposition \ref{prop4.1} does not hold in general. Thus, we can start from a quadratic symplectic Lie superalgebra without the structure of filiform type and obtain a quadratic symplectic one but of filiform type by choosing a suitable superderivation $D$ (even nilpotent skew-supersymmetric).
\end{rem}

Next, we continue with generalized quadratic symplectic double extensions of  quadratic symplectic Lie superalgebras of filiform type by the one-dimensional Lie superalgebra with the even part zero.

\begin{prop} \label{prop4.2}
Assume that  $\displaystyle ({\gg},B,\omega)$ is a quadratic
symplectic Lie superalgebra of filiform type with $dim({\gg}_{\bar 1})=m > 0$ and such that ${\gg}_{\bar 1}$ is a  \textit{filiform ${\gg}_{\bar 0}$-module}  with a flag ${\gg}_{\bar 1}=V_m \supset \dots \supset V_1 \supset V_0.$ Set that $V_1= \KK e_1$ and 
$V_m \backslash V_{m-1}=\mathbb{K}e_m$. Let $\displaystyle \delta$ be the unique even invertible  skew-supersymmetric superderivation of $\displaystyle
(\gg,B)$ such that $\displaystyle  \omega (X,Y)= B(\delta(X),Y)$,
$\displaystyle   \forall \: { X,Y\in \gg}$. Consider
$\displaystyle D$ an odd skew-supersymmetric superderivation of
$\displaystyle (\gg,B)$ and $\displaystyle X_0$ a non-zero
ele\-ment of $\displaystyle {\gg}_{\bar{ 0}}$ such that $$D(X_0) = 0, \;  B(X_0,X_0)= 0, \; D^2  = \frac{1}{2}[X_0,\cdot]_{\gg}, \mbox{ and } e_m \in D(\gg_{\bar 0}).$$

Let $(\widetilde{\gg} =\KK e \oplus \gg \oplus \KK e^*,\widetilde{B})$ be the generalized double
extension of $(\gg,B)$ by the $1$-dimensional Lie superalgebra $\displaystyle (\KK e)_{\bar 1}$ (by means of $D$ and $\displaystyle X_0$). If there exist
$\displaystyle \alpha \in \mathbb{K} \backslash \{0\}$,
$\displaystyle \mu \in \mathbb{K}$, and $\displaystyle A_1 \in
\gg_{\bar 1}$ such that
\begin{eqnarray}
\displaystyle  && [\delta,D]+\alpha D =  {\rm ad}_{A_1},
\label{eq:nhju113}\\
&&  D(A_1) = \alpha X_0 + \frac{1}{2} \delta(X_0),
\label{eq:nhju114}
\end{eqnarray}
then the linear endomorphism  $\displaystyle \widetilde{\delta}: \widetilde{\gg}\longrightarrow \widetilde{\gg}$ defined by
\begin{eqnarray*}
\displaystyle && \widetilde{\delta}(e^\ast) \;=\;  \alpha e^\ast,\\
\hspace*{0cm}&& \widetilde{\delta}(e) \;=\; \mu e^\ast+A_1-\alpha e,\\
\hspace*{0cm}&&\widetilde{\delta}(X) \;=\;
\delta(X)-B(X,A_1)e^\ast,\hspace*{0.5cm}\forall \: { X \in \gg},
\end{eqnarray*}
is an even invertible skew-supersymmetric superderivation of
$\displaystyle (\widetilde{\gg},\widetilde{B}) $. Then, by defining $\displaystyle  \widetilde{\omega} (X,Y):=
\widetilde{B}(\widetilde{\delta}(X),Y), \forall X,Y\in \widetilde{\gg}$ we obtain  $\displaystyle (\widetilde{\gg},\widetilde{B},\widetilde{\omega})$ which is a quadratic
symplectic Lie superalgebra of filiform type, i.e. ${\widetilde{\gg}}_{\bar 1}$ is a \textit{filiform ${\widetilde{\gg}}_{\bar 0}$-module}  with the flag defined by: $${\widetilde{\gg}}_{\bar 1} = {\widetilde{V}}_{m+2} \supset \dots \supset {\widetilde{V}}_1
\supset {\widetilde{V}}_0,$$ where
$ {\widetilde{V}}_{m+2} = \mathbb{K}e \oplus V_{m} \oplus \mathbb{K}e^\ast$, 
$ {\widetilde{V}}_i = V_{i-1}\oplus \mathbb{K}e^\ast$, for $1\leq i\leq m+1$, and ${\widetilde{V}}_0=\{0\}$. 

In this case, the quadratic symplectic Lie superalgebra $\displaystyle (\widetilde{\gg},\widetilde{B},\widetilde{\omega})$ is called the gene\-ralized quadratic symplectic double extension\index{generalized quadratic symplectic double!extension of a quadratic symplectic Lie superalgebra} of $\displaystyle (\gg,B,\omega)$ by the one-dimensional Lie superalgebra
$\displaystyle (\mathbb{K}e)_{{\bar 1}}$ (by means of $\delta$, $\displaystyle X_{0} $, $\displaystyle A_1$, $\alpha$, and $\mu$). \label{teo:asdfg234}
\end{prop}

\begin{proof} The result derives from  Theorem $6.5$ of \cite{quadraticFLSA}  together with Theorem $4.9$ of \cite{simplecticLSA}.  
\end{proof}

\begin{exa}
Consider the complex quadratic symplectic Lie superalgebras of filiform type $(\gg_{4,1}^s, B, \omega)$ with $\gg_{4,1}^s$  and $B$ as described in Proposition \ref{prop3.1} and from all the possible $\omega$ we consider the infinitely many $\omega$ that follows
$$\omega(X_0,X_1)=2b_1, \quad \omega(Y_2,Y_1)=-b_1,  \quad b_1 \in \CC, \ b_1 \neq 0.$$
That is, for each value of $b_1$ we consider $\displaystyle  \omega (X,Y)= B(\delta(X),Y)$,
$\displaystyle   \forall \: { X,Y\in \gg_{4,1}^s}$ with $\delta$ the even invertible skew-supersymmetric superderivation defined by:
$$\delta(X_1)=-2b_1X_1, \ \delta(Y_1)=b_1Y_1, \ \delta(X_0)=2b_1X_0, \ \delta(Y_2)=-b_1Y_2, \ b_1 \neq 0.$$

We consider the odd skew-supersymmetric superderivations $D$ defined by 
$$D(X_1)=a_1Y_1, \quad D(X_0)=0, \quad D(Y_1)=0, \quad D(Y_2)=a_1X_0, \ a_1 \neq 0.$$
Let us note that all them verify the conditions required in  order to be able to apply the generalized double extension, i.e. $  D(X_0)  =  0, \;
  B(X_0,X_0) = 0, \;
 D^2  = \frac{1}{2}[X_0,.]_{\gg}$, in particular $D^2$ is the null derivation which is the same as $\frac{1}{2}ad_{X_0}$ since $X_0$ is a central element of the superalgebra. Moreover we have the extra condition for having a double extension of filiform type, i.e. $e_m \in D({\gg}_{{\bar 0}})$ since for us $e_m=Y_1$. It can be easily checked that $\alpha=-3b_1$, $A_1=Y_2$ and $a_1=-2b_1$ verify 
 $$[\delta,D]+\alpha D = \mbox{ad } (A_1) \quad \mbox{ and } \quad D(A_1) = \alpha X_0 + \frac{1}{2} \delta(X_0).$$
Thus, we have a one-parameter family of quadratic symplectic Lie superalgebras of filiform type: 
$$\begin{array}{l}
[e,e]=X_0 \\{}
[e,X_1]=-[X_1,e]=2b_1Y_1-e^{*}  \\{}
[e,Y_2]=[Y_2,e]=-2b_1 X_0 \\{}
[X_1,Y_1]=-[Y_1,X_1]=-2Y_2 \\{}
[X_1,Y_2]=-[Y_2,X_1]=-2b_1e^{*} \\{}
[Y_1,Y_1]=-2X_0
 \end{array} $$
 This family is of filiform type with respect to the  flag :
$$ V_4 \supset V_3 \supset V_2 \supset V_1
\supset
 V_0$$
 where $V_4=span_{\CC}\{e, Y_1,Y_2,e^*\}$,  $V_3=span_{\CC}\{ Y_1,Y_2,e^*\}$, $V_2=span_{\CC}\{ Y_2,e^*\}$,  $V_1=span_{\CC}\{ e^* \}$ and $V_0=\{ 0\}$. Moreover, the even invertible skew-supersymmetric superderivation of
$\displaystyle(\widetilde{\gg},\widetilde{B})$ is defined by
\begin{eqnarray*}
\displaystyle && \widetilde{\delta}(e^\ast) \;=\; -3b_1 e^\ast,\\
\hspace*{0cm}&& \widetilde{\delta}(e) \;=\;  \mu e^* +Y_2+3b_1e,\\
\hspace*{0cm}&&\widetilde{\delta}(X_0) \;=\;
 2b_1X_0,\\
 \hspace*{0cm}&&\widetilde{\delta}(X_1) \;=\;
 -2b_1X_1,\\
 \hspace*{0cm}&&\widetilde{\delta}(Y_1) \;=\;
 b_1Y_1+e^*,\\
 \hspace*{0cm}&&\widetilde{\delta}(Y_2) \;=\;
 -b_1Y_2,
\end{eqnarray*}
\end{exa}

Let us show now the converse of Proposition \ref{prop4.2}.
 
\begin{thm}\label{pro:convasdfg234}
Let $(\gg,B,\omega)$ be a quadratic symplectic Lie superalgebra of filiform type.
\begin{itemize}
\item[(i)] If $\dim \gg_{\bar 1}=2$ then $(\gg,B,\omega)$ is an elementary odd double extension of the quadratic  symplectic    Lie algebra by the one-dimensional Lie superalgebra $(\mathbb{K}e)_{\bar 1}$.
\item[(ii)] If $\dim \gg_{\bar 1}>2$ then $(\gg,B,\omega)$ is the generalized double extension of a quadratic symplectic Lie superalgebra of a filiform type $(\tilde{\gg},\tilde{B},\tilde{\omega})$ ($\dim \tilde{\gg} = \dim \gg-2$) by the one-dimensional Lie superalgebra $(\mathbb{K}e_m)_{\bar 1}$.
\end{itemize}
\end{thm}

\begin{proof}
(i) If $\dim \gg_{\bar 1} = 2$, then $$\gg_{\bar 1} = V_2 = \mathbb{K}e_1 \oplus \mathbb{K}e_2 \supset V_1 = \mathbb{K}e_1 \supset V_0 = \{0\}.$$  By Lemma \ref{lem1}, we have $V_1 = \mathfrak{z}(\gg) \cap \gg_{\bar 1}$. Since $[\gg_{\bar 0},\gg_{\bar 1}]=\mathbb{K}e_1$ and $\mathbb{K}e_1 \subset \mathfrak{z}(\gg)$ then  $B(e_1,e_1)=0$, by invariance of $B$. Let us note $X_0 := [e_2,e_2]$.
Let $X  $ be any element in  $ \gg_{\bar 0}$. We have
$$[X,X_0] =  \bigl[X,[e_2,e_2]\bigr] = \bigl[[X,e_2],e_2\bigr] + \bigl[e_2,[X,e_2]\bigr] = 0$$ because $[X,e_2] \in V_1$.
By Jacobi identity we also have $$\bigl[e_2,[e_2,e_2]\bigr] = 0,$$  and since $e_1 \in \mathfrak{z}(\gg)$, $$\bigl[e_1,[e_2,e_2]\bigr] = 0,$$ so we conclude that $[e_2,e_2] \in \gg_{\bar 0} \cap \mathfrak{z}(\gg)$. We have $[\gg_{\bar 1},\gg_{\bar 1}]=\mathbb{K}[e_2,e_2]$ because $e_1 \in \mathfrak{z}(\gg)$. Since $B([\gg_{\bar 1},\gg_{\bar 1}],\gg_{\bar 0}) = B(\gg_{\bar 1},[\gg_{\bar 0},\gg_{\bar 1}])$ we get $[\gg_{\bar 1},\gg_{\bar 1}]\neq \{0\}$ due to $[\gg_{\bar 0},\gg_{\bar 1}]=\mathbb{K}e_1 \neq \{0\}$. Therefore $[e_2,e_2]\neq 0$. Moreover, by invariance of $B$ we have $B(X_0,X_0)=0$. Then $(\gg,B)$ is the elementary odd double extension of the quadratic Lie algebra $(\mathfrak{h} := V_1^{\perp} / V_1,\overline{B})$, where $\overline{B}(\overline{X},\overline{Y}):=B(X,Y)$ for all $X,Y \in V_1^{\perp}$, by the one-dimensional Lie superalgebra $\mathbb{K}e_2$ with null product (by means of $X_0$).

By Lemma \ref{lem1}, we have $V_1 = \mathfrak{z}(\gg) \cap \gg_{\bar 1}$, and hence $\delta (V_1) = V_1$ and $\delta (e_1) = \alpha e_1$, where $\alpha \in \mathbb{K} \setminus \{0\}$.
Since $\delta$ is skew-symmetric with respect of $B$, then $\delta(V_1^{\perp}) = V_1^{\perp}$. Consequently $\overline{\delta} : \mathfrak{h} \to \mathfrak{h}$ given by $\overline{\delta}(\overline{X}):=\delta(X)$, for $\overline{X} \in \mathfrak{h}$, is well-defined and $\overline{\delta}$ is an invertible derivation of $\mathfrak{h}$ such that $\overline{\delta}$ is skew-symmetric with respect of $\overline{B}$.
Therefore $(\mathfrak{h} := V_1^{\perp} / V_1, \overline{B}, \overline{\omega})$ is a quadratic symplectic Lie algebra, where $\overline{\omega} : \mathfrak{h} \times \mathfrak{h} \to \mathbb{K}$ defined as $$\overline{\omega}(\overline{X},\overline{Y}) := \overline{B}(\overline{\delta}(\overline{X}), \overline{Y}) = B(\delta(X),Y) = \omega(X,Y),$$ for $\overline{X}, \overline{Y} \in  \mathfrak{h}$.
Since $\delta$ is even, we get $\delta(e_2)=\mu e_2 + \nu e_1$, with $\mu ,\nu \in \mathbb{K}$. Consequently,
\begin{align*}
\delta(X_0) &= \delta([e_2,e_2])\\
&= [\delta(e_2),e_2] + [e_2,\delta(e_2)]\\
&= 2[\delta(e_2),e_2]\\
&= 2[\mu e_2 + \nu e_1, e_2]\\
&= 2\mu [e_2,e_2]\\
&= 2\mu X_0.
\end{align*}
The fact that $B(\delta(e_1),e_2)=-B(e_1,\delta(e_2))$ implies that $\alpha =-\mu$.
It is clear that $(\gg,B,\omega)$ is also the odd double extension of $(\mathfrak{h}, \overline{B}, \overline{\omega})$ by the one-dimensional Lie superalgebra  $\mathbb{K}(e_2 + \frac{\nu}{2\mu} e_1)$ with null product (by means of $\overline{X}_0$).

Let us remark that if we note $$e:=e_2 + \frac{\nu}{2\mu} e_1, \hspace{2cm} e^*:=e_1, \hspace{2cm} \beta := 2\mu,$$ then $$\delta(e) = \frac{\beta}{2} e, \hspace{2cm} \delta(e^*) = -\frac{\beta}{2} e^*, \hspace{2cm} \delta(X_0) = \beta X_0.$$

Applying Theorem \ref{elementary odd double extion}, we conclude that $(\gg,B,\omega)$ is the elementary odd double extension of the quadratic symplectic Lie algebra $(V_1^{\perp} / V_1, \overline{B}, \overline{\omega})$ by the one-dimensional Lie superalgebra $(\mathbb{K}e)_1$ (by means of $X_0$ and $\beta$).

(ii) Let us assume that $(\gg, B, \omega)$ is a quadratic symplectic Lie superalgebra of filiform type, i.e. $\gg_{\bar 1}$ has the structure of filiform $\gg_{\bar 0}$-module with the flag $$\gg_{\bar 1} = V_m \supset V_{m-1} \supset \dots \supset V_1 \supset V_0,$$ and we denote $V_1:=\KK e_1$.
By  Lemma \ref{lem1}  we know that  $V_1 = {\mathfrak z}(\gg) \cap \gg_{\bar 1}$. 
By  Lemma $2.13$ of \cite{quadraticFLSA},  we have $B(\gg_{\bar 0} \oplus V_{m-1}, e_1) = \{0\}$ and as $B$ is nondegenerate, then there exists  $e_m \in V_m \backslash V_{m-1}$  such that $B(e_1,e_m)\neq 0$. We may assume that $B(e_1,e_m) = 1$ (we recall that as $B \mid_{\gg_{\bar 1} \times \gg_{\bar 1}}$ is skew-symmetric then $B(e_m,e_m)=0$ and $B(e_1,e_1)=0$). 
Let us consider ${\mathfrak h}: = (\KK e_1\oplus \KK e_m)^\perp$, that is, ${\mathfrak h} = {\mathfrak h}_{\bar 0} \oplus {\mathfrak h}_{\bar 1}$ is the  $\mathbb{Z}_2$-graded vector space orthogonal to $\KK e_1 \oplus \KK e_m$ with respect to $B$. It comes that $\widetilde{B} := {B} \mid_{{\mathfrak h}\times {\mathfrak h}}$ is non-degenerate. Then $\displaystyle {\gg}_{\bar 0}\subseteq {\mathfrak h}$ and $\displaystyle \gg = {\gg}_{\bar
0}\oplus \bigl(\mathbb{K}e_m \oplus {\mathfrak h}_{\bar 1} \oplus \mathbb{K}e_1 \bigr)$, because $\displaystyle {\gg} = {\mathfrak h}\oplus {\mathfrak h}^\perp$.

Let $\delta$ be the unique even  invertible skew-supersymmetric superderivation of $(\gg,B)$ such that $\omega (X,Y) = B(\delta(X),Y)$, $\displaystyle \forall \: {X,Y\in \gg}$.  Since $\delta({\mathfrak z}(\gg) \cap {\gg}_{\bar 1}) = {\mathfrak z}(\gg)\cap \gg_{\bar 1}$ and ${\mathfrak z}(\gg) \cap \gg_{\bar 1} = V_{\bar 1} = span_{\mathbb K}\{e_1\}$ (Lemma \ref{lem1}) then there exists $\alpha \in \mathbb{K} \setminus \{0\}$ such that $\delta(e_1) = \alpha e_1$.  We denote the graded ideal $\displaystyle I = \mathbb{K}e_1$. As $\displaystyle B(e_1,e_1)= 0$, then $\displaystyle I\subseteq J$, where $J$ is the orthogonal of $I$ with respect to $B$. As $I$ is stable under $\delta$, we have that $J$ is also stable under $\delta$. It can be seen that $\displaystyle J = \mathbb{K}e_1 \oplus {\mathfrak h}$ and therefore $\displaystyle {\mathfrak h}$ is a graded vector subspace of $\displaystyle \gg$ contained in the graded ideal $\displaystyle J$, thus 
\begin{eqnarray*}
\displaystyle [X,Y] &=& \alpha (X,Y) + \varphi (X,Y)e_1,
\hspace*{0.5cm}\forall \: {X, Y \in {\mathfrak h}}, 
\end{eqnarray*}
with $\displaystyle \alpha (X,Y) \in {\mathfrak h}$ and $\displaystyle
 \varphi (X,Y) \in \mathbb{K}$. And,
\begin{eqnarray*}
\displaystyle [e_m,X] &=& D (X) + \psi(X)e_1,
\hspace*{0.5cm}\forall \: { X \in {\mathfrak h}},
\end{eqnarray*}
with $\displaystyle D(X) \in {\mathfrak h}$ and $\displaystyle \psi(X) \in \mathbb{K}$.\\
\\
In \cite{quadraticFLSA}, it is proved that $\displaystyle ({\mathfrak h},\alpha
=[\cdot,\cdot]_{\mathfrak h}, \widetilde{B})$ is a quadratic Lie superalgebra of filiform type.
After we have proved the filiform type part, all the rest of the proof follows from \cite[Theorem 4.10]{simplecticLSA} with $e^*=e_1$ and $e=e_m$.
\end{proof}

From  Theorem \ref{elementary odd double extion}, where is presented the notion of elementary odd double extension of quadratic
symplectic Lie algebras,  and Corollary $4.12$ of \cite{simplecticLSA} we can assert
an inductive description of quadratic symplectic Lie superalgebras of filiform type.
\begin{cor}
Any quadratic symplectic Lie superalgebra $\displaystyle (\gg,B,\omega)$ of filiform type is obtained from a quadratic symplectic Lie algebra by an elementary odd double extension followed by a sequence of generalized quadratic symplectic double extension by the
one-dimensional Lie superalgebra $\displaystyle (\mathbb{K}e)_{{\bar
1}}$.\label{cor:vgfyty}
\end{cor}

\textbf{Acknowledgment:}  The authors would like to thank the reviewers for their comments and suggestions that assist
to the authors in improving the presentation of the paper.

\end{document}